\documentclass{article} 
 
\usepackage[utf8]{inputenc} 

\usepackage{geometry}
\usepackage{graphicx} 

 \usepackage{array} 

\usepackage{mathtools}
\usepackage{amsfonts,hyperref,url,textcomp}
\usepackage{amssymb}
\usepackage{amsthm,xcolor}
\usepackage{amsmath,bbm,bm}
\providecommand{\keywords}[1]{\textbf{\textit{Keywords: }} #1}

\newcommand{\ba}{\begin{eqnarray}}
\newcommand{\ea}{\end{eqnarray}}

\newcommand{\cadlag}{{c\`adl\`ag }}

\newcommand*\tcircle[1]{%
 \raisebox{-0.5pt}{%
  \textcircled{\fontsize{7pt}{0}\fontfamily{phv}\selectfont #1}%
 }%
}

  \newtheorem{theorem}{Theorem}[section] 
  \newtheorem{lemma}[theorem]{Lemma}
  \newtheorem{corollary}[theorem]{Corollary}
  
  \newtheorem{proposition}[theorem]{Proposition}

  \theoremstyle{definition}
  \newtheorem{definition}[theorem]{Definition}

\theoremstyle{remark}
\newtheorem{remark}[theorem]{Remark}

\newtheorem{theorem*}{Theorem}
\newtheorem{lemma*}[theorem]{Lemma}
\newtheorem{corollary*}[theorem]{Corollary}
\newtheorem{proposition*}[theorem]{Proposition}
\newtheorem{problem*}[theorem]{Problem}
\newtheorem{conjecture*}[theorem]{Conjecture}

\newtheorem{examplex}{Example}
\newenvironment{example}
 {\pushQED{\qed}\examplex}
 {\popQED\endexamplex}

\everymath{\displaystyle}

\begin{document}
\title{Quadratic variation along refining partitions: \\Constructions and Examples}
\author{Rama CONT and Purba DAS\footnote{Email: das@maths.ox.ac.uk}
}

\date{Mathematical Institute, University of Oxford} 
\maketitle

\begin{abstract}
We present several constructions of paths and processes with finite quadratic variation along a  refining sequence of partitions, extending previous constructions to the non-uniform case. We study in particular the dependence of  quadratic variation with respect to the sequence of partitions for these constructions. We identify a class of paths whose quadratic variation along a partition sequence is invariant under {\it coarsening}. This class is shown to include typical sample paths of Brownian motion, but also  paths which are  $\frac{1}{2}$-H\"older continuous. Finally, we show how to extend these constructions to higher dimensions.
\end{abstract}
\tableofcontents

\section{ Introduction  }

The concept of quadratic variation of a path along a sequence of partitions, introduced by F\"ollmer \cite{follmer1981},  plays an important role in pathwise Ito calculus \cite{ananova2017,follmer1981,davis2018} and its extensions to path-dependent functionals \cite{CF2010,chiu2020}. 
Examples of functions with  (non-zero) finite quadratic variation are given by typical sample paths of Brownian motion and semi-martingales,  but explicit constructions of such functions have also been  given  by Gantert \cite{gantert1994}, Schied \cite{schied2016} and Mishura and Schied \cite{schied2016b}, in the spirit of Takagi's construction \cite{takagi1901}.
These constructions are based on a Faber-Schauder representation associated with a dyadic sequence of partitions and exploit certain identities which result  from the dyadic nature of the construction.

It is well known that for semimartingales and, more generally, Dirichlet processes \cite{follmer1981b}, quadratic variation, defined as a limit in probability, is invariant with respect to the choice of the partition sequence as long as it has vanishing step size. Almost-sure convergence results for quadratic variation have been obtained for specific classes of processes either under conditions on mesh size (see e.g. \cite{dudley1973,dudley2011})  or for {\it refining partitions} without any conditions on the mesh size \cite{levy1940}. These results do not assume any specific partition sequence and allow for non-uniform partitions.
On the other hand, it is well known \cite{das2020,davis2018,schied2016} that the quadratic variation of a function along a sequence of partitions is not invariant with respect to the choice of this sequence. Conditions for such an invariance to hold have been studied in \cite{das2020} but some of the aforementioned constructions, based on the dyadic partition,  do not fulfil these conditions. The question therefore arises whether such constructions may be carried out for non-dyadic and, more generally, non-uniform partitions sequences and whether the quadratic variation of the resulting functions is invariant with respect to the partition sequence.

We investigate these questions by providing several constructions of paths and processes with finite quadratic variation along  refining sequences of partitions, extending previous constructions to the case of non-uniform partitions. The construction relies on a Schauder basis representation associated with the partition sequence.
We study in particular the dependence of  quadratic variation with respect to the sequence of partitions for these constructions. We identify a class of paths whose quadratic variation along a partition sequence is invariant under {\it coarsening} of the partition sequence. This class is shown to include typical sample paths of Brownian motion, but also  paths which are  $\frac{1}{2}$-H\"older continuous. Finally, we show how to extend these constructions to higher dimensions.

\paragraph{Outline}
Section \ref{sec:definitions} recalls the definition of quadratic variation along a sequence of partitions, following \cite{chiu2018,follmer1981}. 
In Section \ref{sec2}, we construct a Haar basis and Schauder system associated with an arbitrary (finitely) refining partition sequence and recall some properties of the Schauder representation of continuous functions  (Proposition \ref{coeff_hat_func}). Section \ref{sec3} extends the results of Gantert \cite{gantert1994} to the case of a finitely refining (non-uniform) partition sequence and presents some explicit calculations and pathwise estimates. 
In Section \ref{Extended.BM.Schied}, we construct a class of processes with a {\it prescribed} quadratic variation along  an arbitrary finitely refining partition $\pi$ of $[0,1]$, extending the construction \cite{schied2016}  beyond the dyadic case. 
Section  \ref{sec5} discusses the dependence of quadratic variation with respect to the partition sequence. Theorem \ref{main.theorem} provides an example of a class of continuous processes with finite quadratic variation along a finitely refining partition $\pi$ whose quadratic variation is invariant  under  coarsening  of the partitions (Definition \ref{def.coarsening}). Typical Brownian paths are shown to belong to this class.
Finally,  Section \ref{sec6} discusses extensions of these constructions to higher dimensions.

\section{Quadratic variation along a sequence of partitions}\label{sec:definitions}

Let $T>0$. We denote
$D([0,T],\mathbb{R}^d)$ the space of $\mathbb{R}^d$-valued right-continuous functions with left limits (c\`adl\`ag functions),
 $C^{0}([0,T],\mathbb{R}^d)$ the subspace of continuous functions and,
for $0< \nu < 1,$ 
$C^{\nu}([0,T],\mathbb{R}^d)$ the space of H\"older continuous functions with exponent $\nu$:
$$ C^{\nu}([0,T],\mathbb{R}^d) =\left\{ x\in C^0([0,T],\mathbb{R}^d)\quad \Big|\quad \sup_{(t,s)\in[0,T]^2, t\neq s }\frac{\|x(t)-x(s)\|}{|t-s|^{\nu}}<+\infty  \right\},$$ $${\rm and}\qquad C^{\nu-}([0,T],\mathbb{R}^d) =\mathop{\bigcap}_{0\leq \alpha< \nu}C^{\alpha}([0,T],\mathbb{R}^d). $$ 

We denote by $\Pi([0,T])$; the set of all finite partitions of $[0,T]$.
A sequence of partitions of $[0,T]$ is a sequence $(\pi^n)_{n\geq 1}$ of elements of $\Pi([0,T])$:
$$\pi^n=\left(0=t^{n}_0<t^n_1<\cdots<t^{n}_{N(\pi^n)}=T\right).$$ 
We denote $N(\pi^n)$ the number of intervals in the partition $\pi^n$ and
\ba\underline{\pi^n}=\inf_{i=0,\cdots, N(\pi^n)-1} |t^n_{i+1}-t^n_i|,\qquad |\pi^n|=\sup_{i=0,\cdots, N(\pi^n)-1} |t^n_{i+1}-t^n_i|. \ea the size of the largest (respectively the smallest) interval of $\pi^n$.

\begin{example} Let $k\geq 2$ be an integer. The $k$-adic partition sequence of $[0,T]$ is defined by $$ \pi^n=\bigg(t^n_j= \frac{j \ T}{k^n}, \qquad j=0,\cdots, k^n\bigg).$$We have $ \underline{\pi^n}=|\pi^n|= T/k^{n}.$\end{example}

\begin{example}[Lebesgue partition]Given $x\in D([0,T], \mathbb{R}^d)$ define $$ \lambda^n_0(x)=0,\quad{\rm and}\ \forall k\geq 1; \quad \lambda^n_{k+1}(x)=\inf\{ t\in ( \lambda^n_k(x),T],\quad \|x(t)-x( \lambda^n_k(x))\| \geq 2^{-n} \}$$and $N(\lambda^n(x))=\inf\{k\geq 1, \quad \lambda^n_k(x)=T \}.$We call the sequence $\lambda^n(x)=( \lambda^n_k(x) )$ the (dyadic) Lebesgue partition associated to $x$.\end{example}

\begin{definition}[Quadratic variation of a path along a sequence of partitions]\label{def.pathwiseQV}
Let $\pi^n=\left(0=t^{n}_0<t^n_1<\cdots<t^{n}_{N(\pi^n)}=T\right)$ be a sequence of partitions  of  $[0,T]$ with vanishing mesh
$|\pi^n|\to 0$.
A c\`adl\`ag function $x \in D([0,T],\mathbb{R})$ is said to have finite
quadratic variation along the sequence of partitions
$(\pi^n)_{n\geq 1}$ if 
 the sequence of  measures
 \begin{equation*}
    \sum_{t^n_j \in \pi^n} (x(t^n_{j+1})-x(t^n_j))^2 \delta_{t^n_j}\quad 
   \end{equation*}
where $\delta_{t^n_j}$ denotes a unit point mass at $t^n_j$,  converges weakly on $[0,T]$ to a Radon measure $\mu$ such that
$t\mapsto [x]_\pi^c(t) = \mu([0,t]) - \sum_{0 < s \leq t} |\Delta x(s)|^2$
 is continuous and increasing.
The increasing function $[x]_\pi:[0,T]\to \mathbb{R}_+$ defined by \ba [x]_\pi(t)=\mu([0, t])=\lim_{n\to\infty} \sum_{\pi_n} (x(t^n_{k+1}\wedge t)-x(t^n_k\wedge t))^2  \label{eq:qv}\ea  is  called the {\it quadratic variation of }  $x$ along the sequence of partitions $\pi$. We denote $Q_\pi([0,T] ,\mathbb{R})$ the set of c\`adl\`ag paths with these  properties.
\end{definition}

$Q_\pi([0,T] ,\mathbb{R})$ is not a vector space (see e.g \cite{schied2016}). 
The extension to vector-valued paths requires some care \cite{follmer1981}:
\begin{definition}[Pathwise quadratic variation for a vector valued path]\label{defn.qv.vector}
A c\`adl\`ag path  $x=(x^1,...,x^d)\in D([0,T],\mathbb{R}^d)$ is said to have finite quadratic variation along $\pi=(\pi^n)_{n\geq 1}$ if for all $i,j=1,\cdots,d$ we have
$x^i\in Q_\pi([0,T] ,\mathbb{R})$ and $x^i+x^j\in Q_\pi([0,T] ,\mathbb{R})$. We then denote $[x]_\pi\in D([0,T], S^+_d)$ the matrix-valued function defined by,
$$[x]_\pi^{i,j}(t)=\frac{[x^i+x^j]_\pi(t)-[x^i]_\pi(t)-[x^j]_\pi(t)}{2}$$
where
$S^+_d$ is the set of symmetric semi-definite positive matrices.
We denote by $Q_\pi([0,T] ,\mathbb{R}^d)$ the set of functions satisfying these properties.
\end{definition}
For $x\in Q_\pi([0,T] ,\mathbb{R}^d)$, $[x]_\pi$ is a \cadlag function with values in $S^+_d$: $[x]_\pi\in D([0,T],S^+_d)$.

As shown in \cite{chiu2018}, the above definitions may be more simply expressed in terms of convergence of discrete approximations.
For continuous paths, we have the following characterization \cite{cont2012,chiu2018} for quadratic variation:
\begin{proposition}\label{prop.cont.qv}
$x\in C^0([0,T],\mathbb{R}^d)$ has finite quadratic variation along a partition sequence $\pi=(\pi^n,n\geq 1)$ if and only if the sequence of functions $\left([x]_{\pi^n}, \; n\geq 1\right)$ defined by
$$[x]_{\pi^n}(t):=\sum_{t^n_j\in \pi^n}\left(x(t^n_{j+1}\wedge t)-x(t^n_{j}\wedge t)\right)^t \left(x(t^n_{j+1}\wedge t)-x(t^n_{j}\wedge t)\right),$$
converges uniformly on $[0,T]$ to a continuous (non-decreasing) function $[x]_{\pi} \in C^0([0,T],S^+_d)$.
\end{proposition} 
 
The notion of quadratic variation along a sequence of partitions is different from the p-variation for $p=2$. The p-variation involves taking a supremum  over {\it all} partitions, whereas quadratic variation is a limit taken along a specific partition sequence $(\pi^n)_{n\geq 1}$. In general $[x]_\pi$ given by \eqref{eq:qv} is smaller than the p-variation for $p=2$. In fact, for diffusion processes, the typical situation is that p-variation is (almost-surely) infinite for $p=2$ \cite{dudley2011,taylor1972} while the quadratic variation is finite for sequences satisfying some mesh size condition. 
For instance, typical paths of Brownian motion have finite quadratic variation along any sequence of partitions with mesh size $o(1/\log n)$ \cite{dudley1973,delavega1974} while simultaneously having infinite p-variation almost surely for $p\leq 2$ 
\cite[p. 190]{levy1965} :
$$ \inf_{\pi\in \Pi(0,T)} \sum_{\pi} |W(t_{k+1})- W(t_k)|^2 = 0,\qquad {\rm while}\qquad\sup_{\pi\in \Pi(0,T)} \sum_{\pi} |W(t_{k+1})- W(t_k)|^2 = \infty$$
almost-surely.

Definition \ref{def.pathwiseQV} is sensitive to the choice of the partition sequence and is not invariant with respect to this choice, as discussed \cite{davis2018,das2020}. This dependence of quadratic variation with respect to the choice of the partition sequence is discussed in detail in \cite{das2020}. We will come back to this point in our examples below, especially in Section \ref{sec5}.

\section{Schauder system associated with a finitely refining partition sequence}\label{sec2}
 
 The constructions in \cite{gantert1994,schied2016,schied2016b} made use of the 
 Haar basis \cite{haar1910} and Faber-Schauder system \cite{schauder,semadeni1982} associated with a dyadic partition sequence.
 
This is a commonly used tool, but they are constructed  along dyadic partitions. There are current literatures on non-uniform Haar wavelets extensions \cite{francois2003}, but they do not generate an orthonormal basis, as in the uniform case. In this section, firstly we introduce the class of {\it finitely refining} partition sequences which can be thought of a branching process with finite branching at every level (locally), but does not process any global bound on the ratio of partition sizes. Then we construct an orthonormal `non-uniform' Haar basis and a corresponding Schauder system along any finitely refining sequence of partitions.
\subsection{Sequences of interval partitions}
\begin{definition}[Refining sequence of partition]
A sequence of partitions $\pi=(\pi^n)_{n\geq 1}$ of $[0,T]$ with \[ \pi^n= \left(0=t^n_1<t^n_2<\cdots<t^n_{N(\pi^n)}=T\right),\] 
is said to be a refining (or nested) sequence of partitions if 
\[\text{ for all } n\geq 1, \quad t\in \pi^m \implies t\in \cap_{n=m}^{\infty} \pi^n.\]
\end{definition}
In particular $\pi^1\subseteq \pi^2 \subseteq \cdots$. 
Now we introduce a subclass of refining partitions that have a `finite branching' property at every level. 
\begin{definition}[Finitely refining sequence of partitions]\label{finite.refining}
We call a sequence of partitions $\pi$ of $[0,T]$ to be a finitely refining sequence of partitions if $\pi$ is refining with mesh $|\pi^n|\to 0$ and  there exists $ M < \infty$ such that the number of partition points of $\pi^{n+1}$ within any two consecutive partition points of $\pi^n$ is always bounded above by $M$, irrespective of $n\in \mathbb{N}$. 
\end{definition}
For a finitely refining sequence of partitions $\pi$, there exists $M< \infty$ such that $\sup_n \frac{N(\pi^n)}{M^n}\leq 1$.
A subsequence of a finitely refining sequence may not be a finitely refining sequence but has to be a refining sequence. This property ensures the partition has locally finite branching at every step but do not ensure any global bound on partitions size. This is ensured by the following property \cite{das2020}:
\begin{definition}[Balanced partition sequence]\label{def.balance}
Let $\pi^n=\left(0=t^{n}_0<t^n_1<\cdots <t^{n}_{N(\pi^n)}=T\right)$ be a sequence of partitions of $[0,T]$. Then we say $\pi=(\pi^n)_{n\geq 1}$ is {\em balanced} if 
\begin{equation}\text{there exists } \; c>0, \quad \text{such that, for all }\;  n\geq 1, \quad \frac {|\pi^n|}{\underline{\pi^n}}\leq c.\label{eq.balance}\end{equation} 
\end{definition}
The balanced condition for partition means that all intervals in the partition $\pi^n$ are asymptotically comparable. Note that since $\underline{\pi^n}N(\pi^n)\leq T$, any balanced sequence of partitions also satisfies
\begin{equation}
|\pi^n|\leq c \ \underline{\pi^n} \leq \frac{cT}{N(\pi^n)}.\label{eq.wellbalanced}
\end{equation}
If a sequence of partitions $\pi$ of $[0,T]$ is finitely refining and balanced at the same time (for example dyadic/uniform partition) then $$\limsup_{n} \frac{|\pi^n|}{\underline{\pi^{n+1}}}<\infty.$$
\begin{definition}[complete refining partition]
A sequence of partitions $\pi=(\pi^n)_{n\geq 1}$ of $[0,1]$ is said to be complete refining if there exists positive constants $\epsilon$ and $M$ such that:
\[\text{ for all } n\geq 1, \quad 1+\epsilon \leq \frac{|\pi^n|}{|\pi^{n+1}|}\leq M.\]
\end{definition}

\subsection{Haar basis associated with a finitely refining partition sequence}
Let $\pi$ be a finitely refining sequence of refining partition of $[0,1]$ 
\[\pi^n= \left(0=t^n_0<t^n_1<\cdots<t^n_{N(\pi^n)}=1\right) \]
with   mesh $|\pi^n|\to 0$.
  Now define $p(n,k)$ as follows:
 \[p(n,k) = \inf\{j\geq 0 \; :\;t^{n+1}_j\geq t^n_k \}. \]
 Since $\pi$ is refining   the following inequality holds:
\begin{equation}\label{EqFor_p}
\text{for all }\; k=0,\cdots, N(\pi^n)-1, \qquad 0\leq t^{n}_{k}= t^{n+1}_{p(n,k)}<t^{n+1}_{p(n,k)+1}<\cdots <t^{n+1}_{p(n,k+1)}= t^{n}_{k+1}\leq 1.    
\end{equation}
We now define the Haar basis associated with such as partition sequence:
\begin{definition}[Haar basis]
The   Haar basis associated with a  finitely refining partition sequence $\pi=(\pi^n)_{n\geq 1}$ is a collection of   piece-wise constant functions $\{\psi_{m,k,i}, m=0,1,\cdots, k=0,\cdots,N(\pi^m)-1, i = 1,\cdots, p(m,k+1)-p(m,k)\}$ defined as follows:
\ba\label{haar_basis}
\psi_{m,k,i}(t)= 
\begin{cases}
  \quad 0 &\quad\text{if } t\notin \left[t^{m+1}_{p(m,k)},t_{p(m,k)+i}^{m+1}\right) \\
  
  \left( \frac{t^{m+1}_{p(m,k)+i}-t^{m+1}_{p(m,k)+i-1}}{t^{m+1}_{p(m,k)+i-1} - t^{m+1}_{p(m,k)}}\times\frac{1}{t^{m+1}_{p(m,k)+i}-t^{m+1}_{p(m,k)}} \right)^{\frac{1}{2}} &\quad\text{if } t\in\left[t_{p(m,k)}^{m+1},t_{p(m,k)+i-1}^{m+1}\right) \\
  
  -\left( \frac{t^{m+1}_{p(m,k)+i-1} - t^{m+1}_{p(m,k)}}{t^{m+1}_{p(m,k)+i}-t^{m+1}_{p(m,k)+i-1}}\times\frac{1}{t^{m+1}_{p(m,k)+i}-t^{m+1}_{p(m,k)}}\right)^{\frac{1}{2}} &\quad\text{if } t\in \left[t_{p(m,k)+i-1}^{m+1},t_{p(m,k)+i}^{m+1}\right).
  \end{cases} 
\ea
Note, $t^{m+1}_{p(m,k)+i-1}\in \pi^{m+1}\backslash\pi^{m}$ for all $i$ and $t^{m+1}_{p(m,k)}=t^m_k \in \pi^m \cap \pi^{m+1}$. Since $\pi$ is a finitely refining sequence of partitions $p(m,k+1)-p(m,k)\leq M <\infty$, for all $m,k$.
\end{definition}
For any finitely refining partition $\pi$, the  family of functions $\{\psi_{m,k,i}\}_{m,k,i}$ can be reordered as $\{\psi_{m,k}\}_{m,k}$. For each level $m\in \{0,1,\cdots\}$, the values of $k$ runs from $0$ to $N(\pi^{m+1})-N(\pi^m)-1$ (after reordering).

The following properties are easily derived from the definition:
\begin{proposition}\label{Haar.properties}
The non-uniform Haar basis along a finitely refining sequence of partitions $\pi=(\pi^n)_{n\geq 1}$ has the following properties:
\\(i). For fixed $m\in \{0\}\cup \mathbb{N}$, the piece-wise constant functions $\psi_{m,k,i}(t)$ and $\psi_{m,k',i'}(t)$ have disjoint supports for all $k\neq k'\in \{0,1,\cdots, N(\pi^m)-1\}$ and for all $i,i'$.
\\(ii). For fixed $m\in \{0\}\cup \mathbb{N}$ and fixed $k$, the support of the piece-wise constant function $\psi_{m,k,i}(t)$ is contained in the support of $\psi_{m,k,i'}(t)$ as soon as $i\leq i'$.
\\(iii). For all $m\in \{0\}\cup \mathbb{N}$, for all $k\in \{0,1,\cdots, N(\pi^m)-1\}$ and for all $i$
\[\int_{\mathbb{R}} \psi_{m,k,i}(t)dt =\int_{0}^1 \psi_{m,k,i}(t)dt = 0. \]
(iv). Orthogonality:
\[\int_{\mathbb{R}} \psi_{m,k,i}(t)\psi_{m',k',i'}(t)dt =\int_{0}^1 \psi_{m,k,i}(t)\psi_{m',k',i'}(t)dt = \mathbbm{1}_{m,m'}\mathbbm{1}_{k,k'}\mathbbm{1}_{i,i'}, \]
where $\mathbbm{1}_{a,b}$ is $1$ if $a=b$ and $0$ otherwise.
\end{proposition}
 As a consequence of (iii) and (iv), 
the family $\{\psi_{m,k,i}; \quad \forall m,k,i\}$  is an orthonormal family.

\subsection{Schauder representation of a continuous function}
The Schauder basis functions $e^{\pi}_{m,k,i}$ are obtained by integrating the  Haar basis functions:
\[e^{\pi}_{m,k,i}: [0,1]\to \mathbb{R} \; \text{ where,} \quad e^{\pi}_{m,k,i}(t)= \int_0^t \psi_{m,k,i}(s)ds =\left(\int_{t^{m+1}_{p(m,k)}}^{t\wedge t^{m+1}_{p(m,k)+i}} \psi_{m,k,i}(s)ds\right) \mathbbm{1}_{[t^m_k,t^{m+1}_{p(m,k)+i}]}. 
\]
For all $m,k,i$ the functions $e^\pi_{m,k,i}:[0,1]\to \mathbb{R}$ are continuous but not differentiable and 
{\footnotesize
\ba\label{defn.e}
 e^{\pi}_{m,k,i}(t) =
\begin{cases}
  \quad 0 &\quad\text{if } t\notin \left[t_{p(m,k)}^{m+1},t_{p(m,k)+i}^{m+1}\right) \\
  
  \left( \frac{t^{m+1}_{p(m,k)+i}-t^{m+1}_{p(m,k)+i-1}}{t^{m+1}_{p(m,k)+i-1} - t^{m+1}_{p(m,k)}}\times\frac{1}{t^m_{p(m,k)+i}-t^m_{p(m,k)}} \right)^{\frac{1}{2}}\times(t-t^{m+1}_{p(m,k)}) &\quad\text{if } t\in\left[t_{p(m,k)}^{m+1},t_{p(m,k)+i-1}^{m+1}\right) \\
  
  \left( \frac{t^{m+1}_{p(m,k)+i-1} - t^{m+1}_{p(m,k)}}{t^{m+1}_{p(m,k)+i}- t^{m+1}_{p(m,k)+i-1}}\times\frac{1}{t^{m+1}_{p(m,k)+i}- t^{m+1}_{p(m,k)}}\right)^{\frac{1}{2}}\times(t^{m+1}_{p(m,k)+i}-t) &\quad\text{if } t\in \left[t_{p(m,k)+i-1}^{m+1},t_{p(m,k)+i}^{m+1}\right)
  \end{cases}.
\ea 
}

Assume that $x\in C^0([0,1], \mathbb{R})$ is a continuous function with the following Schauder representation along a finitely refining sequence of partitions $\pi$:
\[ x(t) = a_0+a_1t+ \sum_{m=0}^{\infty} \sum_{k=0}^{N(\pi^{m+1})-N(\pi^m)-1} \theta_{m,k} e^{\pi}_{m,k}(t), \quad \]
where, for all $ m,k$, the coefficients $ a_0,a_1,\theta_{m,k}\in \mathbb{R}$; are constants.   Denote by $x^N(t):[0,1]\to \mathbb{R}\in C^0([0,1],\mathbb{R})$ the linear interpolation of $x$ along partition points of $\pi^N$:
\[ x^N(t) = a_0+a_1t+ \sum_{m=0}^{N-1} \sum_{k=0}^{N(\pi^{m+1})-N(\pi^m)-1} \theta_{m,k} e^{\pi}_{m,k}(t). \quad \]
\begin{lemma}\label{prop.partition.points1}
For all $N\geq 2$, for all $t\in \pi^N$ one have, $x(t) = x^N(t)$.
\end{lemma}
\begin{proof}
From the construction of $e^{\pi}_{m,k}$, for all $m\geq N,$ and   for all $k$, we have $e^{\pi}_{m,k}(t^N_{i}) = 0$. So for $t\in \pi^N$ :
\[ x(t) = \sum_{m=0}^{\infty} \sum_{k=0}^{N(\pi^{m+1})-N(\pi^m)-1} \theta_{m,k} e^{\pi}_{m,k}(t)=\sum_{m=0}^{N-1} \sum_{k=0}^{N(\pi^{m+1})-N(\pi^m)-1} \theta_{m,k} e^{\pi}_{m,k}(t)= x^N(t). \]
\end{proof}
If the  sequence of partitions $\pi$ has vanishing mesh then as a limit the continuous function $x^N$ converges to $x\in C^0([0,1],\mathbb{R})$ in uniform norm 
\[\lim_{N\to \infty} \sup_{t\in [0,1]}\left| x^N(t)- x(t) \right| =0.\]

\begin{theorem} \label{coeff_hat_func}
Let $\pi$ be a finitely refining sequence of partitions of $[0,T]$. Then any $x\in C^0([0,1],\mathbb{R})$  has a unique Schauder representation:
\[ x(t) =x(0) +(x(1)-x(0))t+ \sum_{m=0}^{\infty} \sum_{k=0}^{N(\pi^{m+1})-N(\pi^{m})-1} \theta_{m,k} e^{\pi}_{m,k}(t).\]
If the support of the function $e_{m,k}^\pi$ is $[t^{m,k}_1,t^{m,k}_3]$ and  its maximum is attained at time $t^{m,k}_2$ then, the coefficient $\theta_{m,k}$ has a closed form representation as follows:
\begin{equation}\label{eq.theta.coeff}
 \theta_{m,k} = \frac{\bigg[\left(x(t^{m,k}_{2})-x(t^{m,k}_{1})\right)(t^{m,k}_{3}-t^{m,k}_{2})-\left(x(t^{m,k}_{3})- x(t^{m,k}_{2})\right)(t^{m,k}_{2}-t^{m,k}_1) \bigg]}{\sqrt{(t^{m,k}_2-t^{m,k}_1)(t^{m,k}_3-t^{m,k}_2)(t^{m,k}_3-t^{m,k}_1)}}.
\end{equation}
\end{theorem}
\begin{proof}
Take the function $y$ as $y(t)= x(t)- x(0) + (x(0)-x(1))t$. Since $x$ is a continuous function, so does $y$. Also for the function $y$ we have $y(0)=y(1)=0$. So without loss of generality we will assume $x(0)=x(1)=0$ for the rest of the proof.
\par Since $t^{m,k}_1,t^{m,k}_2, t^{m,k}_3 \in \pi^{m+1}$, using Proposition \ref{prop.partition.points1} we get:
\[ \quad x(t^{m,k}_1)= x^{m+1}(t^{m,k}_1),\quad x(t^{m,k}_2)=x^{m+1}(t^{m,k}_{2})\quad \text{and } \; x(t^{m,k}_3)=x^{m+1}(t^{m,k}_3). \]
Now we can write the increment $ x(t^{m,k}_2)-x(t^{m,k}_1)$ as follows.
\[x(t^{m,k}_2)-x(t^{m,k}_1) = \left(x^{m+1}(t^{m,k}_2)-x^{m+1}(t^{m,k}_1)\right)\]
\[= \sum_{n=0}^{m} \sum_{\{k: \psi_{n,k}(t^{m,k}_1)\neq 0\}}\theta_{n,k}\times \psi_{n,k}(t^{m,k}_{1}) \times (t^{m,k}_2-t^{m,k}_{1}), \]
where $k$ is such that for which the function $\psi_{n,k}$ has strictly positive value in the interval $(t^{m,k}_{1},t^{m,k}_2)$. Now one can notice that for all $n<m$, $\psi_{n,k(.)}(t^{m,k}_{1}) = \psi_{n,k(.)}(t^{m,k}_2)$. So for the expansion of weighted second difference $\left(x(t^{m,k}_2)-x(t^{m,k}_1)\right)\left(t^{m,k}_3-t^{m,k}_2\right)-\left(x(t^{m,k}_3)- x(t^{m,k}_2)\right)$ $\left( t^{m,k}_2-t^{m,k}_1\right)$, all values cancel out except for the term involving $\theta_{m,k}$. So we get the following identity:
\[\left(x(t^{m,k}_2)-x(t^{m,k}_1)\right)(t^{m,k}_3-t^{m,k}_2)-\left(x(t^{m,k}_3)- x(t^{m,k}_2)\right)(t^{m,k}_2-t^{m,k}_1) )\]
\[ = \theta_{m,k}\left[\psi_{m,k}(t^{m,k}_1) \times (t^{m,k}_2-t^{m,k}_1)(t^{m,k}_3- t^{m,k}_2)-\psi_{m,k}(t^{m,k}_2) \times (t^{m,k}_3-t^{m,k}_2)(t^{m,k}_2- t^{m,k}_1) \right] \]
\[ = \theta_{m,k} \times (t^{m,k}_3-t^{m,k}_2)(t^{m,k}_2- t^{m,k}_1)\left[\left( \frac{t^{m,k}_3-t^{m,k}_2}{t^{m,k}_2 - t^{m,k}_1}\times\frac{1}{t^{m,k}_3-t^{m,k}_1} \right)^{\frac{1}{2}}+\left( \frac{t^{m,k}_2 - t^{m,k}_1}{t^{m,k}_3-t^{m,k}_2}\times\frac{1}{t^{m,k}_3-t^{m,k}_{1}} \right)^{\frac{1}{2}} \right] \]
\[ = \theta_{m,k} \times \sqrt{(t^{m,k}_3-t^{m,k}_2)(t^{m,k}_2- t^{m,k}_1)} \times \left[\sqrt{t^{m,k}_3-t^{m,k}_1} \right]. \]
Note that the value of $\theta_{m,k}$ only depends on the function $x$ and the partition $\pi$. So the result follows.
\end{proof}

\section{Quadratic variation along finitely refining partitions}\label{sec3}
Gantert \cite{gantert1994} provides a formula for the quadratic variation of a function along the dyadic partition in terms of coefficients in the dyadic Faber-Schauder basis. In this section, we generalize these results to any finitely refining sequence of partitions.

\textit{Notation:} For a function $x\in C^0([0,1],\mathbb{R})$ and a sequence of partitions $\pi$ of $[0,1]$, we denote 
\[[x]_{\pi^n}(t) := \sum_{i=0}^{N(\pi^n)-1} \left(x(t^n_{i+1}\wedge t)-x(t^n_i\wedge t)\right)^2,\]
the quadratic variation of $x$ along $\pi$ at level $n$.
\begin{proposition}\label{QV-weighted-schied}
Let $\pi$ be a finitely refining sequence of partitions  of $[0,1]$ with vanishing mesh and $(e^{\pi}_{m,k})$ be the associated Schauder basis. Let $x\in C^0([0,1],\mathbb{R})$ given by 
\[x(t) = x(0)+\big(x(1)-x(0)\big)t+ \sum_{m=0}^{\infty} \sum_{k=0}^{N(\pi^{m+1})-N(\pi^m)-1} \theta_{m,k} e^{\pi}_{m,k}(t).\]
Then the quadratic variation of $x$ along $\pi_n$ is given by:
\[ [x]_{\pi^n}(t) = \sum_{m=0}^{n-1} \sum_{k=0}^{N(\pi^{m+1})-N(\pi^m)-1} a_{m,k}^n(t)\theta_{m,k}^2 +\sum_{m,m'} \sum_{\substack{k,k'\\(m,k)\neq (m',k')}}b^n_{m,k,m',k'}(t)\theta_{m,k}\theta_{m',k'}  .\] 
Denoting by $[t^{m,k}_1,t^{m,k}_3]$ the support of  $e_{m,k}^\pi$,  $t^{m,k}_2$ the point at which it reaches it maximum and $$ \Delta t^n_i= t^n_{i+1}\wedge t - t^n_i \wedge t,$$ we have the following closed form expression for $a_{m,k}^n(t)$ and $b^n_{m,k,m',k'}(t)$:
\[ a_{m,k}^n(t) =\left\{ \left[ \sum_{  t^n_i \in [t^{m,k}_{1}, t^{m,k}_{2}]} (\Delta t^n_i)^2\right]\times \frac{t^{m,k}_{3}-t^{m,k}_{2}}{t^{m,k}_{2}-t^{m,k}_{1}} + \left[ \sum_{ t^n_i \in [t^{m,k}_{2}, t^{m,k}_{3}]} (\Delta t^n_i)^2\right]\times \frac{t^{m,k}_{2}-t^{m,k}_{1}}{t^{m,k}_{3}-t^{m,k}_{2}}\right\} \times \frac{1}{t^{m,k}_{3}-t^{m,k}_{1}}, \]
\[b^n_{m,k,m',k'}(t) = \psi_{m',k'}(t^{m,k}_1)\times \left\{ \frac{ \sum_{  t^n_i \in [t^{m,k}_{1}, t^{m,k}_{2}]} (\Delta t^n_i)^2}{t^{m,k}_{2}-t^{m,k}_{1}} - \frac{\sum_{ t^n_i \in [t^{m,k}_{2}, t^{m,k}_{3}]} (\Delta t^n_i)^2}{t^{m,k}_{3}-t^{m,k}_{2}}\right\} \times \sqrt{\frac{(t^{m,k}_{2}-t^{m,k}_{1})(t^{m,k}_{3}-t^{m,k}_{2})}{t^{m,k}_{3}-t^{m,k}_{1}}} \]
if ${\rm supp}(e^n_{m,k})\subset {\rm supp}(e^n_{m',k'})$ and $b^n_{m,k,m',k'}(t) =0$ otherwise.
\end{proposition}
\begin{remark}
Similar to the dyadic case \cite{gantert1994}, the coefficients $a^n_{m,k}$ and $b^n_{m,k,m',k'}$   only depend on the sequence of partitions $\pi$ and not on the path $x\in C^0([0,1],\mathbb{R})$.
\end{remark}
\begin{proof}
We compute $[x]_{\pi^n}(1)$. For $t\in [0,1]$, the calculations are analogously done with the stopped path $x(t\wedge .)$.
\[[x]_{\pi^n}(1)= \sum_{i=0}^{N(\pi^n)-1} \left(x(t^n_{i+1})-x(t^n_{i})\right)^2 =\sum_{i=0}^{N(\pi^n)-1} \left(\sum_{m=0}^{n-1}\sum_{\{k: \psi_{m,k}(t^n_{i})\neq 0\}} \theta_{m,k} \left(e^\pi_{m,k}(t^n_{i+1})-e^\pi_{m,k}(t^n_{i})\right)\right)^2\]
\[=\sum_{i=0}^{N(\pi^n)-1} \left(\sum_{m=0}^{n-1}\sum_{\{k: \psi_{m,k}(t^n_{i})\neq 0\}} \theta_{m,k} \int_{t^n_{i}}^{t^n_{i+1}}\psi_{m,k}(u) du\right)^2\]
\[=\sum_{i=0}^{N(\pi^n)-1} \left(\sum_{m=0}^{n-1}\sum_{\{k: \psi_{m,k}(t^n_{i})\neq 0\}} \theta_{m,k}\times\psi_{m,k}(t^n_{i}) (t^n_{i+1}-t^n_{i}))\right)^2. \]
Since $\pi$ is a finitely refining sequence of partitions, there exists an upper bound  $M$ on the cardinality of the set $$\{ k\geq 1,\quad  \psi_{m,k}(t^n_{i})\neq 0\}$$ for any $m\leq n$. So in the above expression of $[x]_{\pi^n}(1)$ if we look at the coefficient of $\theta^2_{m,k}$ for some pair $(m,k)$ we get:
\[ \sum_{\{i:\psi_{m,k}(t^n_{i}) \neq 0\}}\left[\psi_{m,k}(t^n_{i}) (t^n_{i+1}-t^n_{i})\right]^2 \]
\[=\left\{ \left[ \sum_{\Delta t^n_i \subset [t^{m,k}_{1}, t^{m,k}_{2}]} (\Delta t^n_i)^2\right]\times \frac{t^{m,k}_{3}-t^{m,k}_{2}}{t^{m,k}_{2}-t^{m,k}_{1}} + \left[ \sum_{\Delta t^n_i \subset [t^{m,k}_{2}, t^{m,k}_{3}]} (\Delta t^n_i)^2\right]\times \frac{t^{m,k}_{2}-t^{m,k}_{1}}{t^{m,k}_{3}-t^{m,k}_{2}}\right\} \times \frac{1}{t^{m,k}_{3}-t^{m,k}_{1}}. \]
For two pairs $(m,k)$ and $(m',k')$ if $e^n_{m,k}$ and $e^n_{m',k'}$ have disjoint support then $\psi_{m,k}(t)\psi_{m',k'}(t) = 0$ for all $t$, hence coefficient of $\theta_{m,k}\theta_{m',k'}$ is always zero. 
For two pairs $(m,k)$ and $(m',k')$ with ${\rm supp}(e^n_{m,k})\subset$ ${\rm supp}(e^n_{m',k'})$; $\psi_{m',k'}(t)$ is a non-zero constant for all $t$. This is a consequence of the fact $\{\psi_{m,k}\}$ is orthonormal. Now if we look at the coefficient of $\theta_{m,k}\theta_{m',k'}$ for the case  when ${\rm supp}(e^n_{m,k})\subset {\rm supp}(e^n_{m',k'})$, we get:
\[ \sum_{\{i:\psi_{m,k}(t^n_{i}) \neq 0\}}\left[\psi_{m,k}(t^n_{i}) (t^n_{i+1}-t^n_{i})\right]\times \left[\psi_{m',k'}(t^n_{i}) (t^n_{i+1}-t^n_{i})\right] \]
\[= \sum_{\{i:\psi_{m,k}(t^n_{i}) \neq 0\}}\left[\psi_{m,k}(t^n_{i})\psi_{m',k'}(t^{m,k}_1) \right]\times  (t^n_{i+1}-t^n_{i})^2\]
\[ = \psi_{m',k'}(t^{m,k}_1)\times \left\{ \frac{ \sum_{\Delta t^n_i \subset [t^{m,k}_{1}, t^{m,k}_{2}]} (\Delta t^n_i)^2}{t^{m,k}_{2}-t^{m,k}_{1}} - \frac{\sum_{\Delta t^n_i \subset [t^{m,k}_{2}, t^{m,k}_{3}]} (\Delta t^n_i)^2}{t^{m,k}_{3}-t^{m,k}_{2}}\right\} \times \sqrt{\frac{(t^{m,k}_{2}-t^{m,k}_{1})(t^{m,k}_{3}-t^{m,k}_{2})}{t^{m,k}_{3}-t^{m,k}_{1}}}. \]
So the result follows.
\end{proof}
We say  $x\in C^0([0,1],\mathbb{R})$ has bounded Schauder coefficients  along $\pi$ if
\[  \sup_{m,k}|\eta_{m,k}^\pi(x)|< \infty.\]
\par The class of functions $\mathcal{X}$ defined in \cite{schied2016} provide examples of functions with bounded Schauder coefficients (along the dyadic partitions). The following example is an example of continuous function with bounded Schauder coefficients representation along dyadic partition, but quadratic variation along dyadic partition does not exists \cite{schied2016}.
\begin{example}
Consider the sequence $\{\mathbb{T}^n\}_n$ of dyadic partitions and the continuous function $x\in C^0([0,1],\mathbb{R})$ defined as following:
\[x(t) = \sum_{m=0}^{\infty} \sum_{k=0}^{2^m-1} \theta_{m,k}^\mathbb{T} e^{\mathbb{T}}_{m,k}(t), \quad \text{ where, } \theta_{m,k}^\mathbb{T} = 1+(-1)^m.\]
For the function $x$ defined above we have:
\[ [x]_{\mathbb{T}^{2n}}(t) =\frac{4}{3}t \qquad \text{ and, } \; [x]_{\mathbb{T}^{2n+1}}(t) =\frac{8}{3}t. \]
$\mathbb{T}$ is a finitely refining and balanced sequence of partitions with $\frac{|\mathbb{T}^n|}{|\mathbb{T}^{n+1}|} =\frac{\underline{\mathbb{T}^n}}{\underline{\mathbb{T}^{n+1}}}= 2 $.
\end{example}
\begin{theorem}[Quadratic covariation representation]\label{Quadratic.covar}

Let $\pi$ be a finitely refining sequence of partitions of $[0,1]$ with vanishing  mesh  and $(e^\pi_{m,k})$ be  the  associated  Schauder  basis. Let $x,y\in C^0([0,1],\mathbb{R})\cap Q_\pi([0,1],\mathbb{R})$ with unique  representation
\[x(t) = x(0)+(x(1)-x(0))t+ \sum_{m=0}^{\infty} \sum_{k=0}^{N(\pi^{m+1})-N(\pi^m)-1} \theta_{m,k} e^{\pi}_{m,k}(t), \quad \text{ and, }\]
\[y(t) = y(0)+(y(1)-y(0))t+ \sum_{m=0}^{\infty} \sum_{k=0}^{N(\pi^{m+1})-N(\pi^m)-1} \eta_{m,k} e^{\pi}_{m,k}(t).\]
Then, the quadratic covariation of $x$ and $y$ at level $n$ along the sequence of partitions $\pi$ may be represented as:
\[\quad [x,y]_{\pi^n}(t) = \sum_{m=0}^{n-1} \sum_{k=0}^{N(\pi^{m+1})-N(\pi^m)-1} a_{m,k}^n(t)\theta_{m,k}\eta_{m,k}  +\sum_{m,m'} \sum_{\substack{k,k'\\(m,k)\neq (m',k')}}b^n_{m,k,m',k'}(t)\theta_{m,k}\eta_{m',k'}  .\] 
Denoting by $[t^{m,k}_1,t^{m,k}_3]$ the support of  $e_{m,k}^\pi$,  $t^{m,k}_2$ the point at which it reaches it maximum and $$ \Delta t^n_i= t^n_{i+1}\wedge t - t^n_i \wedge t,$$ we have the following closed form expression for $a_{m,k}^n$ and $b^n_{m,k,m',k'}$.
\[ a_{m,k}^n(t) =\left\{ \left[ \sum_{\Delta t^n_i \subset [t^{m,k}_{1}, t^{m,k}_{2}]} (\Delta t^n_i)^2\right]\times \frac{t^{m,k}_{3}-t^{m,k}_{2}}{t^{m,k}_{2}-t^{m,k}_{1}} + \left[ \sum_{\Delta t^n_i \subset [t^{m,k}_{2}, t^{m,k}_{3}]} (\Delta t^n_i)^2\right]\times \frac{t^{m,k}_{2}-t^{m,k}_{1}}{t^{m,k}_{3}-t^{m,k}_{2}}\right\} \times \frac{1}{t^{m,k}_{3}-t^{m,k}_{1}}, \]
and,
\[b^n_{m,k,m',k'}(t) = \psi_{m',k'}(t^{m,k}_1)\times \left\{ \frac{ \sum_{\Delta t^n_i \subset [t^{m,k}_{1}, t^{m,k}_{2}]} (\Delta t^n_i)^2}{t^{m,k}_{2}-t^{m,k}_{1}} - \frac{\sum_{\Delta t^n_i \subset [t^{m,k}_{2}, t^{m,k}_{3}]} (\Delta t^n_i)^2}{t^{m,k}_{3}-t^{m,k}_{2}}\right\} \times \sqrt{\frac{(t^{m,k}_{2}-t^{m,k}_{1})(t^{m,k}_{3}-t^{m,k}_{2})}{t^{m,k}_{3}-t^{m,k}_{1}}}, \]
if ${\rm supp}(e^n_{m,k})\subset {\rm supp}(e^n_{m',k'})$ and $b^n_{m,k,m',k'}(t) =0$ otherwise.
\end{theorem}
\begin{proof}
The proof is  similar to that of Theorem \ref{QV-weighted-schied}.
\end{proof}
We will now derive some bounds on the coefficients $a^n_{m,k}$\footnote{throughout the rest of the paper in some places we wrote $a^n_{m,k}$ for $a^n_{m,k}(1)$} and $b^n_{m,k,m',k'}$\footnote{similarly we wrote $b^n_{m,k,m',k'}$ for $b^n_{m,k,m',k'}(1)$}  which appear in the expression of quadratic variation   in Theorem \ref{QV-weighted-schied}.
\begin{proposition}\label{boundOn_ab}
If $\pi$ is a finitely refining sequence of partitions of $[0,1]$ then 
\[0\leq \underline{\pi^n} \leq a^n_{m,k}\leq |\pi^n|.\]
If we also assume the sequence of partitions $\pi$ is balanced, then there exists $C>0$ such that  
\[   {\rm supp}(e^n_{m,k})\subset{\rm supp}(e^n_{m',k'})\quad \Rightarrow \; 0\leq |b^n_{m,k,m',k'}|\leq C(|\pi^n|-\underline{\pi^n})\sqrt{\frac{|\pi^m|}{|\pi^{m'}|}}.\]
If ${\rm supp}(e^n_{m,k})\cap{\rm supp}(e^n_{m',k'})=\emptyset$ then $b^n_{m,k,m',k'} = 0$. 
\end{proposition}
\begin{proof}From Theorem \ref{QV-weighted-schied} we have the expression of $a^n_{m,k}$ as follows:
\[ a_{m,k}^n =\left\{ \left[ \sum_{\Delta t^n_i \subset [t^{m,k}_{1}, t^{m,k}_{2}]} (\Delta t^n_i)^2\right]\times \frac{t^{m,k}_{3}-t^{m,k}_{2}}{t^{m,k}_{2}-t^{m,k}_{1}} + \left[ \sum_{\Delta t^n_i \subset [t^{m,k}_{2}, t^{m,k}_{3}]} (\Delta t^n_i)^2\right]\times \frac{t^{m,k}_{2}-t^{m,k}_{1}}{t^{m,k}_{3}-t^{m,k}_{2}}\right\} \times \frac{1}{t^{m,k}_{3}-t^{m,k}_{1}}.\]
Since $\underline{\pi^n}\times(\Delta t^n_i)\leq (\Delta t^n_i)^2 \leq |\pi^n|\times(\Delta t^n_i)$, for all $m,k$  we can bound $a_{m,k}^n$ as follows.
\[a_{m,k}^n \leq |\pi^n|\left[\left\{ \left[ \sum_{\Delta t^n_i \subset [t^{m,k}_{1}, t^{m,k}_{2}]} \Delta t^n_i\right] \frac{t^{m,k}_{3}-t^{m,k}_{2}}{t^{m,k}_{2}-t^{m,k}_{1}} + \left[ \sum_{\Delta t^n_i \subset [t^{m,k}_{2}, t^{m,k}_{3}]} \Delta t^n_i\right] \frac{t^{m,k}_{2}-t^{m,k}_{1}}{t^{m,k}_{3}-t^{m,k}_{2}}\right\} \times \frac{1}{t^{m,k}_{3}-t^{m,k}_{1}} \right] \]
\[=|\pi^n|\times\left[(t^{m,k}_{3}-t^{m,k}_{2})+(t^{m,k}_{2}-t^{m,k}_{1}) \right]\frac{1}{t^{m,k}_{3}-t^{m,k}_{1}} = |\pi^n|. \]
Similarly using the other side of the inequality, we get for all $n,m,k$: $\underline{\pi^n}\leq a_{m,k}^n\leq |\pi^n|.$ So the first part of the result follows. For the second part of the proposition, for any $(m,k)$, under the balanced assumption on $\pi$, we have $|\psi_{m,k}(t)|\leq C_1\sqrt{\frac{1}{|\pi^m|}}$. So under balanced assumption:
\[|b^n_{m,k,m',k'}| = \left|\psi_{m',k'}(t^{m,k}_1)\right|\times \left| \frac{ \sum_{\Delta t^n_i \subset [t^{m,k}_{1}, t^{m,k}_{2}]} (\Delta t^n_i)^2}{t^{m,k}_{2}-t^{m,k}_{1}} - \frac{\sum_{\Delta t^n_i \subset [t^{m,k}_{2}, t^{m,k}_{3}]} (\Delta t^n_i)^2}{t^{m,k}_{3}-t^{m,k}_{2}}\right| \times \sqrt{\frac{(t^{m,k}_{2}-t^{m,k}_{1})(t^{m,k}_{3}-t^{m,k}_{2})}{t^{m,k}_{3}-t^{m,k}_{1}}} \]
\[\leq C_2\sqrt{\frac{1}{|\pi^{m'}|}} \left| \frac{ \sum_{\Delta t^n_i \subset [t^{m,k}_{1}, t^{m,k}_{2}]} (\Delta t^n_i)^2}{t^{m,k}_{2}-t^{m,k}_{1}} - \frac{\sum_{\Delta t^n_i \subset [t^{m,k}_{2}, t^{m,k}_{3}]} (\Delta t^n_i)^2}{t^{m,k}_{3}-t^{m,k}_{2}}\right| \times \sqrt{|\pi^m|}. \]
Since $\underline{\pi^n}\times(\Delta t^n_i)\leq (\Delta t^n_i)^2 \leq |\pi^n|\times(\Delta t^n_i)$, for all $(m,k)\neq (m',k')$ we can bound $|b_{m,k,m',k'} ^n|$ as follows.
\[|b_{m,k,m',k'}^n|\leq C_3\sqrt{\frac{1}{|\pi^{m'}|}} \left| \frac{ |\pi^n|\sum_{\Delta t^n_i \subset [t^{m,k}_{1}, t^{m,k}_{2}]} (\Delta t^n_i)}{t^{m,k}_{2}-t^{m,k}_{1}} - \frac{\underline{\pi^n}\sum_{\Delta t^n_i \subset [t^{m,k}_{2}, t^{m,k}_{3}]} (\Delta t^n_i)}{t^{m,k}_{3}-t^{m,k}_{2}}\right| \times \sqrt{|\pi^m|} \]
\[\leq C(|\pi^n|-\underline{\pi^n})\sqrt{\frac{|\pi^m|}{|\pi^{m'}|}}. \]

\end{proof}
As a consequence of Proposition \ref{boundOn_ab}, for any uniform partition (such as dyadic  partition), $b^n_{m,k,m',k'} = 0$ for all $m,k,m',k',n\geq 0$.  But since $b^n_{m,k,m',k'}$ are not necessarily positive, if the individual $b^n_{m,k,m',k'}$ are not equal to zero, still $\sum_{m,k,m',k'} b_{m,k,m',k'}^n$ can converge to $0$ as $n\to \infty$.

\begin{lemma}\label{almost.unif}
Consider a balanced finitely refining sequence $\pi$ of partitions  satisfying $t^n_{i+1}-t^n_i = \frac{1}{N(\pi^n)}(1+\epsilon^n_i)$ with $\sup_i |\epsilon^n_i| = o(\frac{1}{n})$ for all $n\geq 1$. Then  any function $x\in C^0([0,1],\mathbb{R})$ with bounded Schauder representation  
\[x(t) =x(0)+(x(1)-x(0))t+ \sum_{m=0}^{\infty} \sum_{k=0}^{N(\pi^{m+1})-N(\pi^m)-1} \theta_{m,k} e^{\pi}_{m,k}(t),\]
we have
\[ [x]_{\pi^n}(t)= \sum_{m=0}^{n-1} \sum_{k=0}^{N(\pi^{m+1})-N(\pi^m)-1} a_{m,k}^n(t)\theta_{m,k}^2. \] 
If $[t^{m,k}_1,t^{m,k}_3]$ the support of  $e_{m,k}^\pi$,  $t^{m,k}_2$ the point at which it reaches it maximum and $$ \Delta t^n_i= t^n_{i+1}\wedge t - t^n_i \wedge t,$$ then: 
\[ a_{m,k}^n(t) =\left\{ \left[ \sum_{\Delta t^n_i \subset [t^{m,k}_{1}, t^{m,k}_{2}]} (\Delta t^n_i)^2\right]\times \frac{t^{m,k}_{3}-t^{m,k}_{2}}{t^{m,k}_{2}-t^{m,k}_{1}} + \left[ \sum_{\Delta t^n_i \subset [t^{m,k}_{2}, t^{m,k}_{3}]} (\Delta t^n_i)^2\right]\times \frac{t^{m,k}_{2}-t^{m,k}_{1}}{t^{m,k}_{3}-t^{m,k}_{2}}\right\} \times \frac{1}{t^{m,k}_{3}-t^{m,k}_{1}}. \]
\end{lemma}
\textit{Note:} The above assumption   is true for any uniform partition $\pi$, say dyadic or triadic partition as in this case $b^n_{m,k,m',k'} = 0$ for all $m,m',k,k'$. But Lemma \ref{almost.unif} does not require having $b^n_{m,k,m',k'} = 0$.
\begin{proof}
We compute $[x]_{\pi^n}(1)$. For $t \in [0,1]$, the calculations are analogously done with the
stopped path $x(t \wedge.)$. \par For any pair $(m,k)$, under the balanced assumption on $\pi$ we have; $|\psi_{m,k}(t)|\leq C_1\sqrt{\frac{1}{|\pi^m|}}$, where constant $C_1$ is independent of $m$ and $k$. We will show that the second term on the quadratic variation formula in Theorem \ref{QV-weighted-schied}: $\sum_{m,m'} \sum_{\substack{k,k'\\(m,k)\neq (m',k')}}b^n_{m,k,m',k'}\theta_{m,k}\theta_{m',k'}$ goes to $0$ as $n\to \infty$.
From the construction of $b^n_{m,k,m',k'}$ we know that if support of 
 $e^n_{m,k}$ and support of $e^n_{m',k'}$ are disjoint, then: $b^n_{m,k,m',k'}=0$. So,
 $$\sum_{m,m'} \sum_{\substack{k,k'\\(m,k)\neq (m',k')}}b^n_{m,k,m',k'}\theta_{m,k}\theta_{m',k'}$$
\[=\sum_{m=0}^{n-1}\sum_{k=0}^{N(\pi^{m+1})-N(\pi^m)-1}\sum_{m'=0}^{m} \sum_{\substack{k':\text{ Support of }e^n_{m,k} \subset e^n_{m',k'} \\(m,k)\neq (m',k')}}b^n_{m,k,m',k'}\theta_{m,k}\theta_{m',k'} \]
\[\leq M\sum_{m=0}^{n-1}\sum_{k=0}^{N(\pi^{m+1})-N(\pi^m)-1}\sum_{m'=0}^{m}|\theta_{m,k}\theta_{m',k'(.)}| \times\psi_{m',k'(.)}(t^{m,k}_1)\times \sqrt{\frac{(t^{m,k}_{2}-t^{m,k}_{1})(t^{m,k}_{3}-t^{m,k}_{2})}{t^{m,k}_{3}-t^{m,k}_{1}}}\]
\[\times \left| \frac{ \sum_{\Delta t^n_i \subset [t^{m,k}_{1}, t^{m,k}_{2}]} (\Delta t^n_i)^2}{t^{m,k}_{2}-t^{m,k}_{1}} - \frac{\sum_{\Delta t^n_i \subset [t^{m,k}_{2}, t^{m,k}_{3}]} (\Delta t^n_i)^2}{t^{m,k}_{3}-t^{m,k}_{2}}\right| \]
under the balanced assumption on $\pi$:
\[\leq C_2\sum_{m=0}^{n-1}\sum_{k=0}^{N(\pi^{m+1})-N(\pi^m)-1}\sum_{m'=0}^{m}|\theta_{m,k}\theta_{m'k'(.)}| \times\sqrt{\frac{1}{|\pi^{m'}|}} \times \sqrt{|\pi^m|}\]
\[\times \left| \frac{ \sum_{\Delta t^n_i \subset [t^{m,k}_{1}, t^{m,k}_{2}]} (\Delta t^n_i)^2}{t^{m,k}_{2}-t^{m,k}_{1}} - \frac{\sum_{\Delta t^n_i \subset [t^{m,k}_{2}, t^{m,k}_{3}]} (\Delta t^n_i)^2}{t^{m,k}_{3}-t^{m,k}_{2}}\right| \]

\[\leq C_3\sum_{m=0}^{n-1}\sum_{k=0}^{N(\pi^{m+1})-N(\pi^m)-1}\sum_{m'=0}^{m} \left| \frac{ \sum_{\Delta t^n_i \subset [t^{m,k}_{1}, t^{m,k}_{2}]} (\Delta t^n_i)^2}{t^{m,k}_{2}-t^{m,k}_{1}} - \frac{\sum_{\Delta t^n_i \subset [t^{m,k}_{2}, t^{m,k}_{3}]} (\Delta t^n_i)^2}{t^{m,k}_{3}-t^{m,k}_{2}}\right|. \]
The last inequality follows from the fact that $x$ has a bounded Schauder basis representation along a refining sequence of partitions $\pi$ and $\sqrt{\frac{|\pi^m|}{|\pi^{m'}|}}\leq 1 $ for all $m'\leq m$. So the above inequality will reduce as following:
\[\leq C_3\sum_{m=0}^{n-1}\sum_{k=0}^{N(\pi^{m+1})-N(\pi^m)-1}\sum_{m'=0}^{m} \left| \frac{|\pi^n| \sum_{\Delta t^n_i \subset [t^{m,k}_{1}, t^{m,k}_{2}]} (\Delta t^n_i)}{t^{m,k}_{2}-t^{m,k}_{1}} - \frac{\underline{\pi^n}\sum_{\Delta t^n_i \subset [t^{m,k}_{2}, t^{m,k}_{3}]} (\Delta t^n_i)}{t^{m,k}_{3}-t^{m,k}_{2}}\right| \]
\[= C_3\sum_{m=0}^{n-1}\sum_{k=0}^{N(\pi^{m+1})-N(\pi^m)-1}\sum_{m'=0}^{m} \left| |\pi^n| - \underline{\pi^n} \right| \leq C_4\sum_{m=0}^{n-1}\sum_{k=0}^{N(\pi^{m+1})-N(\pi^m)-1}\sum_{m'=0}^{m} \left| \frac{1}{N(\pi^n)} \sup_i |\epsilon^n_i| \right| \]
\[\leq C_4 \left| \frac{1}{N(\pi^n)} \sup_i |\epsilon^n_i| \right|\sum_{m=0}^{n-1} (m+1)\left[N(\pi^{m+1})-N(\pi^m)\right] \leq C_5\times n \sup_i |\epsilon^n_i|\to 0.\]
So the lemma follows.
\end{proof}

\section{Processes with prescribed quadratic variation along a finitely refining partition sequence}\label{Extended.BM.Schied}

\subsection{Processes with linear quadratic variation}
 A well-known example of process with linear quadratic variation i.e. constant quadratic variation per unit time is Brownian motion, which satisfies this property almost surely along any refining partition.
Schied \cite{schied2016} provided a subclass $\mathcal{X}$ of $Q_{\mathbb{T}}([0,1],\mathbb{R})$, such that 
for all $x\in \mathcal{X}$, the  quadratic variation along the dyadic partition is  $[x]_{\mathbb{T}}(t) = t$. 
 However  Brownian motion is not included in the class $\mathcal{X}$ given in \cite{schied2016}. 
 
In this subsection, for any fixed finitely refining sequence of partitions $\pi$, we construct a class $\mathcal{B}^\pi$ of processes with  linear  quadratic variation along  $\pi$ and we show that Brownian motion belongs to $\mathcal{B}^\pi$. With some additional conditions on the sequence of partitions, we also provide an almost sure convergence result. The class $\mathcal{X}$ defined in  \cite{schied2016} has a non-empty intersection with  $\mathcal{B}^\mathbb{T}$.

\par Let $W$ be a Wiener process on a probability space $(\Omega, {\cal F},\mathbb{P})$, which we take to be the canonical Wiener space without loss of generality i.e. $\Omega=C^0([0,T],\mathbb{R}), W(t,\omega)=\omega(t)$.
For finitely refining sequence of partitions $\pi$ of $[0,1]$, the quadratic variation of $W$ along $\pi$ is linear almost surely, i.e. $\forall t\in [0,1]$,\; $\mathbb{P}([W]_{\pi}(t)=t)=1$ \cite{levy1940,levy1965}. On the other hand, $W$ can also be represented in terms of its  Schauder expansion along $\pi$, which provides the following properties of the coefficient.
\begin{lemma}\label{Brownian.coeff}
Let $\pi$ be a finitely refining sequence of partitions and $W$ be a  Brownian motion. Then $W$ has the following    Schauder   expansion along the partition sequence $\pi$:
\[W(t) =W(0)+ (W(1)-W(0))t + \sum_{m=0}^{\infty}\sum_{k=1}^{N(\pi^{m+1})-N(\pi^m)}\eta_{mk}e^{\pi}_{m,k}(t), \]
where $\eta_{m,k}\sim^{IID} N(0,1)$ are independent and identically distributed.
\end{lemma}
\begin{proof}
The projection of Brownian Motion on any basis is always Gaussian, hence $\eta_{m,k}$ is Gaussian. If the support of the function $e_{m,k}^\pi$ is $[t^{m,k}_1,t^{m,k}_3]$ and the maximum is attained at time $t^{m,k}_2$ then, applying Theorem \ref{coeff_hat_func} the coefficient $\eta_{m,k}$ has a closed-form representation as follows.
\begin{equation}\label{eq.theta.coeff1}
  \eta_{m,k} = \frac{\bigg[(W(t^{m,k}_{2})-W(t^{m,k}_{1}))(t^{m,k}_{3}-t^{m,k}_{2})-(W(t^{m,k}_{3})- W(t^{m,k}_{2}))(t^{m,k}_{2}-t^{m,k}_1) \bigg]}{\sqrt{(t^{m,k}_2-t^{m,k}_1)(t^{m,k}_3-t^{m,k}_2)(t^{m,k}_3-t^{m,k}_1)}}.
\end{equation}
Since $W$ is a Brownian motion, \[\mathbb{E}(\eta_{m,k}) = \frac{\bigg[\mathbb{E}(W(t^{m,k}_{2})-W(t^{m,k}_{1}))(t^{m,k}_{3}-t^{m,k}_{2})-\mathbb{E}(W(t^{m,k}_{3})- W(t^{m,k}_{2}))(t^{m,k}_{2}-t^{m,k}_1) \bigg]}{\sqrt{(t^{m,k}_2-t^{m,k}_1)(t^{m,k}_3-t^{m,k}_2)(t^{m,k}_3-t^{m,k}_1)}}=0 \quad \text{ and,}\]
\[Var(\eta_{m,k}) = \frac{Var\left[(W(t^{m,k}_{2})-W(t^{m,k}_{1}))(t^{m,k}_{3}-t^{m,k}_{2})\right]+Var\left[(W(t^{m,k}_{3})- W(t^{m,k}_{2}))(t^{m,k}_{2}-t^{m,k}_1)\right] }{(t^{m,k}_2-t^{m,k}_1)(t^{m,k}_3-t^{m,k}_2)(t^{m,k}_3-t^{m,k}_1)} \]
\[+ \frac{Cov\left((W(t^{m,k}_{2})-W(t^{m,k}_{1}))(t^{m,k}_{3}-t^{m,k}_{2}),(W(t^{m,k}_{3})- W(t^{m,k}_{2}))(t^{m,k}_{2}-t^{m,k}_1) \right)}{(t^{m,k}_2-t^{m,k}_1)(t^{m,k}_3-t^{m,k}_2)(t^{m,k}_3-t^{m,k}_1)}\]
\[=\frac{(t^{m,k}_2-t^{m,k}_1)(t^{m,k}_3-t^{m,k}_2)^2+(t^{m,k}_2-t^{m,k}_1)^2(t^{m,k}_3-t^{m,k}_2)}{(t^{m,k}_2-t^{m,k}_1)(t^{m,k}_3-t^{m,k}_2)(t^{m,k}_3-t^{m,k}_1)} =1. \]
Using the orthogonality of increments of Brownian motion we can show that $Cov\left(\eta_{m,k},\eta_{m'k'}\right) = \mathbbm{1}_{m=m'}\mathbbm{1}_{k=k'}$. Along with the fact that $\eta_{m,k}$ is Gaussian, we can conclude $\eta_{m,k}\sim^{IID} N(0,1)$.
\end{proof}
For Brownian motion $W$ the  quadratic variation along $\pi$ can be represented using the explicit representation of quadratic variation (Theorem \ref{QV-weighted-schied}) as following:
\[[W]_{\pi}(t) = \lim_{n\to\infty}[W]_{\pi^n}(t), \quad \text{with:} \quad
[W]_{\pi^n}(t) = \sum_{m=0}^{n-1}\sum_{k}a^n_{m,k}(t)\eta^2_{m,k} +\sum_{m, m'=0}^{n-1}\sum_{\substack{k,k'\\(m,k)\neq(m',k')}}b^n_{m,k,m',k'}(t)\eta_{m,k}\eta_{m',k'}.\]
Now we know that for Brownian motion $\mathbb{E}[W]_{\pi}(t) =\lim_{n\to\infty} \mathbb{E}[W]_{\pi^n}(t) = t$. So,
\[\lim_{n\to \infty}\left[\mathbb{E}\sum_{m=0}^{n-1}\sum_{k}a^n_{m,k}(t)\eta^2_{m,k}  +\mathbb{E}\sum_{m, m'=0}^{n-1}\sum_{\substack{k,k'\\(m,k)\neq(m',k')}}b^n_{m,k,m',k'}(t)\eta_{m,k}\eta_{m',k'} \right] = t\]
\[\implies\lim_{n\to \infty}\left[\sum_{m=0}^{n-1}\sum_{k}a^n_{m,k}(t)\mathbb{E}\eta^2_{m,k}  +\sum_{m, m'=0}^{n-1}\sum_{\substack{k,k'\\(m,k)\neq(m',k')}}b^n_{m,k,m',k'}(t)\mathbb{E}[\eta_{m,k}\eta_{m',k'}] \right] = t\]
\begin{equation}\label{identity1}
\implies\lim_{n\to \infty}\left[\sum_{m=0}^{n-1}\sum_{k}a^n_{m,k}(t)\right] = t . 
\end{equation}
Since $a^n_{m,k}(t)$ only depends on the refining partition $\pi$, and \textit{not} on the path of Brownian motion, the above invariant is true for any finitely refining sequence of partitions $\pi$. For Brownian motion we also know that $\lim_{n\to \infty} \mathbb{E}([W]_{\pi^n}(t)-t)^2 =0.$ This implies, $\lim_{n\to \infty} \mathbb{E}([W]_{\pi^n}(t))^2 =t^2$. So,
\[\lim_{n\to \infty}\mathbb{E}\left[\sum_{m=0}^{n-1}\sum_{k}a^n_{m,k}(t)\eta^2_{m,k}  +\sum_{m, m'=0}^{n-1}\sum_{\substack{k,k'\\(m,k)\neq(m',k')}}b^n_{m,k,m',k'}(t)\eta_{m,k}\eta_{m',k'} \right]^2 = t^2\]
\[\implies \lim_{n\to \infty}\mathbb{E}\bigg[\sum_{m=0}^{n-1}\sum_{k}(a^n_{m,k}(t))^2\eta^4_{m,k}  +\sum_{m,m'=0}^{n-1}\sum_{\substack{k,k'\\(m,k)\neq(m',k')}}a^n_{m,k}(t)a^n_{m',k'}(t)\eta^2_{m,k}\eta^2_{m',k'}  \]
\[+\sum_{m, m'=0}^{n-1}\sum_{\substack{k,k'\\(m,k)\neq(m',k')}}(b^n_{m,k,m',k'}(t))^2\eta^2_{m,k}\eta^2_{m',k'} \Bigg] = t^2\]
\[\implies \lim_{n\to \infty}\bigg[3\sum_{m=0}^{n-1}\sum_{k}(a^n_{m,k}(t))^2  +\sum_{m,m'=0}^{n-1}\sum_{\substack{k,k'\\(m,k)\neq(m',k')}}a^n_{m,k}(t)a^n_{m',k'}(t) +\sum_{m, m'=0}^{n-1}\sum_{\substack{k,k'\\(m,k)\neq(m',k')}}(b^n_{m,k,m',k'}(t))^2 \Bigg] = t^2\]
\[\implies \lim_{n\to \infty}\left(\sum_{m}^{n-1}\sum_{k}a^n_{m,k}(t) \right)\left(\sum_{m'}^{n-1}\sum_{k'}a^n_{m',k'}(t)\right)\]\[+ \lim_{n\to \infty}\bigg[2\sum_{m=0}^{n-1}\sum_{k}(a^n_{m,k}(t))^2  + \sum_{m, m'=0}^{n-1}\sum_{\substack{k,k'\\(m,k)\neq(m',k')}}(b^n_{m,k,m',k'}(t))^2 \Bigg] = t^2.\]
From Equation \eqref{identity1} we know that the first sum converges to $t^2$. So the above equality reduces to:
\[\lim_{n\to \infty}\bigg[2\sum_{m=0}^{n-1}\sum_{k}(a^n_{m,k}(t))^2 +\sum_{m, m'=0}^{n-1}\sum_{\substack{k,k'\\(m,k)\neq(m',k')}}(b^n_{m,k,m',k'}(t))^2 \Bigg] = 0.\]
Since both the two summations in the limit are positive we get the following two identities:
\begin{equation}\label{identity2}
  \lim_{n\to \infty}\bigg[\sum_{m=0}^{n-1}\sum_{k}(a^n_{m,k}(t))^2 \bigg] =0 \;\text{and, } 
\end{equation}
\begin{equation}\label{identity3}
  \lim_{n\to \infty}\bigg[\sum_{m, m'=0}^{n-1}\sum_{\substack{k,k'\\(m,k)\neq(m',k')}}(b^n_{m,k,m',k'}(t))^2 \Bigg] = 0.
\end{equation}
Since both $a^n_{m,k}(t)$ and $ b^n_{m,k,m',k'}(t)$ are only dependent of the sequence of partitions $\pi$ and \textit{not} dependents on the Brownian path $W$, Equation \eqref{identity2} and Equation \eqref{identity3} are true for any finitely refining sequence of partitions $\pi$ of $[0,1]$.
\par In the following theorem, we provide a class of processes with linear  quadratic variation along  a finitely refining partition sequence $\pi$.
\begin{theorem}\label{mainpart1}
Let $\pi$ be a finitely refining sequence of partitions with vanishing mesh $|\pi^n| \to 0$. Define,  for $t\in [0,1],$
\[x(t) =x(0)+ (x(1)-x(0))t + \sum_{m=0}^{\infty}\sum_{k=1}^{N(\pi^{m+1})-N(\pi^m)}\eta_{mk}e^{\pi}_{m,k}(t),\]
where $\left(\eta_{m,k}, m\in \mathbb{N}, k=1,\cdots,N(\pi^{m+1})-N(\pi^m) \right)$ is a family of random variables with
$$\mathbb{E}\eta_{m,k} = 0, \qquad \mathbb{E}\eta_{m,k}\eta_{m',k'} = \mathbbm{1}_{m,m'}\mathbbm{1}_{k,k'},\qquad \mathbb{E}\eta_{m,k}^4<\infty \quad \text{ and,}$$
$$\mathbb{E}(\eta_{m,k}^\alpha\eta_{m_1,k_1}^\beta\eta_{m_2,k_2}^\gamma\eta_{m_3,k_3}^\delta) =\mathbb{E}(\eta_{m,k}^\alpha)\mathbb{E}(\eta_{m_1,k_1}^\beta)\mathbb{E}(\eta_{m_2,k_2}^\gamma)\mathbb{E}(\eta_{m_3,k_3}^\delta)$$  for all integers $\alpha,\beta,\gamma,\delta$ such that $\alpha+\beta+\gamma+\delta=4$.
Then: 
\[\forall\; \epsilon>0, \qquad \lim_{n\to \infty}\mathbb{P}(|[x]_{\pi^n}(t)-t|>\epsilon)=0.\]
Furthermore, if the sequence of partitions $\pi$ is complete refining and balanced then the quadratic variation of $x$ along $\pi$ \textbf{exists} and is linear almost surely, i.e. 
$$x\in Q_\pi([0,1],\mathbb{R})\; \text{almost surely, and }\;  \mathbb{P}([x]_\pi(t)= t)=1.$$
\end{theorem}
Note that the coefficients are neither assumed independent nor Gaussian, so this class of processes contains examples of processes other  Brownian motions. 
\begin{proof} Using Theorem \ref{QV-weighted-schied} the quadratic variation of $x$ along $\pi$ at level $n$ can be represented as:
\[[x]_{\pi^n}(t) = \sum_{m=0}^{n-1}\sum_{k}a^n_{m,k}(t)\eta^2_{m,k}  +\sum_{m, m'=0}^{n-1}\sum_{\substack{k,k'\\(m,k)\neq(m',k')}}b^n_{m,k,m',k'}(t)\eta_{m,k}\eta_{m',k'} .\]
Now using the assumptions on the coefficient $\eta_{m,k}$, we will show that $\forall t\in[0,1],\; \lim_{n\to \infty}\mathbb{E}[x]_{\pi^n}(t) = t$. 
\[\lim_{n\to\infty}\mathbb{E}[x]_{\pi^n}(t) =\lim_{n\to \infty}\mathbb{E} \left[\sum_{m=0}^{n-1}\sum_{k}a^n_{m,k}(t)\eta^2_{m,k}  +\sum_{m, m'=0}^{n-1}\sum_{\substack{k,k'\\(m,k)\neq(m',k')}}b^n_{m,k,m',k'}(t)\eta_{m,k}\eta_{m',k'} \right]\]
\[ =\lim_{n\to \infty} \left[\sum_{m=0}^{n-1}\sum_{k}a^n_{m,k}(t)\mathbb{E}\eta^2_{m,k}  +\sum_{m, m'=0}^{n-1}\sum_{\substack{k,k'\\(m,k)\neq(m',k')}}b^n_{m,k,m',k'}(t)\mathbb{E}\left(\eta_{m,k}\eta_{m',k'}\right) \right]\]
\[ =\lim_{n\to \infty} \left[\sum_{m=0}^{n-1}\sum_{k}a^n_{m,k}(t) \right] = t. \]
The last equality follows from Equation \eqref{identity1}. Now to prove $[x]_{\pi^n}(t)\to t $ in probability, we only need to show that $ \lim_{n\to \infty}\mathbb{E}([x]_{\pi^n}(t))^2 = t^2$. So:
\[\lim_{n\to \infty}\mathbb{E}\left([x]_{\pi^n}(t)\right)^2 =\lim_{n\to\infty} \mathbb{E}\left[ \sum_{m=0}^{n-1}\sum_{k}a^n_{m,k}(t)\eta^2_{m,k}  +\sum_{m, m'=0}^{n-1}\sum_{\substack{k,k'\\(m,k)\neq(m',k')}}b^n_{m,k,m',k'}(t)\eta_{m,k}\eta_{m',k'} \right]^2 \]
\[= \lim_{n\to \infty}\mathbb{E}\Bigg[\sum_{m=0}^{n-1}\sum_{k}(a^n_{m,k}(t))^2\eta^4_{m,k}  +\sum_{m,m'=0}^{n-1}\sum_{\substack{k,k'\\(m,k)\neq(m',k')}}a^n_{m,k}(t)a^n_{m',k'}(t)\eta^2_{m,k}\eta^2_{m',k'}  \]
\[+\sum_{m, m'=0}^{n-1}\sum_{\substack{k,k'\\(m,k)\neq(m',k')}}(b^n_{m,k,m',k'}(t))^2\eta^2_{m,k}\eta^2_{m',k'}  \]\[+ \sum_{m=0}^{n-1}\sum_{k} \sum_{m', m''=0}^{n-1}\sum_{\substack{k',k''\\(m',k')\neq(m'',k'')}}a^n_{m,k}(t)\eta^2_{m,k} b^n_{m',k',m'',k''}(t)\eta_{m',k'}\eta_{m'',k''}  \Bigg] \]
\[= \lim_{n\to \infty}\left[\sum_{m=0}^{n-1}\sum_{k}(a^n_{m,k}(t))^2 \mathbb{E}\eta_{m,k}^4  +\sum_{m,m'=0}^{n-1}\sum_{\substack{k,k'\\(m,k)\neq(m',k')}}a^n_{m,k}(t)a^n_{m',k'}(t)+\sum_{m, m'=0}^{n-1}\sum_{\substack{k,k'\\(m,k)\neq(m',k')}}(b^n_{m,k,m',k'}(t))^2\right] \]
\[= \lim_{n\to \infty}\left(\sum_{m}^{n-1}\sum_{k}a^n_{m,k}(t)\right)\left(\sum_{m'}^{n-1}\sum_{k'}a^n_{m',k'}(t) \right) \]\[+ \lim_{n\to \infty}\Bigg[\sum_{m=0}^{n-1}\sum_{k}(a^n_{m,k}(t))^2(\mathbb{E}\eta_{m,k}^4-1) + \sum_{m, m'=0}^{n-1}\sum_{\substack{k,k'\\(m,k)\neq(m',k')}}(b^n_{m,k,m',k'}(t))^2 \Bigg]. \]
Using Equation \eqref{identity1} we know that the first sum converges to $t^2$. The last two sum can be bounded above as follows.:
\[\left|\lim_{n\to \infty}\Bigg[\sum_{m=0}^{n-1}\sum_{k}(a^n_{m,k}(t))^2(\mathbb{E}\eta_{m,k}^4-1)+\sum_{m, m'=0}^{n-1}\sum_{\substack{k,k'\\(m,k)\neq(m',k')}}(b^n_{m,k,m',k'}(t))^2\Bigg]\right| \]
\[\leq \lim_{n\to \infty}\Bigg[C\sum_{m=0}^{n-1}\sum_{k}(a^n_{m,k}(t))^2+\sum_{m, m'=0}^{n-1}\sum_{\substack{k,k'\\(m,k)\neq(m',k')}}(b^n_{m,k,m',k'}(t))^2\Bigg]= 0. \]
The last equality follows using the Equality \eqref{identity2} and Equality \eqref{identity3}. So we have $\lim_{n\to \infty}\mathbb{E}[x]_{\pi^n}(t) = t$, and correspondingly $\lim_{n\to \infty}\mathbb{E}\left([x]_{\pi^n}(t) - t \right)^2=0$. So  $[x]_{\pi^n}(t) \to t$ in probability.
\\Now we will prove the almost sure convergence. Since for this part we have already assumed $\pi$ is balanced, from the previous calculations and using the bounds from Proposition \ref{boundOn_ab} we get the bound on $Var([x]_{\pi^n}(t))$ as following:
\[Var([x]_{\pi^n}(t)) \leq \left|C\sum_{m=0}^{n-1}\sum_{k}(a^n_{m,k}(t))^2+\sum_{m, m'=0}^{n-1}\sum_{\substack{k,k'\\(m,k)\neq(m',k')}}(b^n_{m,k,m',k'}(t))^2 \right|\]
\[\leq \left||\pi^n|^2N(\pi^n) + C(|\pi^n|-\underline{\pi^n})^2\sum_{m=0}^{n-1}\sum_{k=0}^{N(\pi^{m+1})-N(\pi^{m})-1}\sum_{m'=0}^{m-1}\frac{|\pi^m|}{|\pi^{m'}|} \right|.\]
Since, $\pi$ is also complete refining there exists $C_0<\infty$ such that $\sum_{m'=0}^{m-1}\frac{|\pi^m|}{|\pi^{m'}|}\leq C_0$. So we get the bound on variance as follows.
\[Var([x]_{\pi^n}(t)) \leq C_1|\pi^n|.\]
Now take $\epsilon_n=|\pi^n|^{\frac{1}{4}}$, then from Markov inequality we have:
\[\mathbb{P}(|[x]_{\pi^n}(t)-t|\geq \epsilon_n)\leq \frac{Var([x]_{\pi^n}(t))}{\epsilon_n^2}\leq C\sqrt{|\pi^n|}. \]
Since $\pi$ is a complete refining sequence of partitions of $[0,1]$, $\sum_{n=0}^\infty \sqrt{|\pi^n|} <\infty$. So using Borel-Cantelli Lemma we can show, $\mathbb{P}(|[x]_{\pi^n}(t)-t|\geq \epsilon_n, \text{ infinitely often }) = 0 $, where $\epsilon=|\pi^n|^{\frac{1}{4}} \to 0$. Hence, we have $[x]_{\pi^n}(t) \to t$ almost-surely. So as a consequence $[x]_\pi(t)= \lim_{n\to \infty} [x]_{\pi^n}(t)$ exists almost surely and $[x]_\pi(t)=t $ almost surely. 
\end{proof}
To summarize, for any finitely refining sequence of partitions $\pi$ we define  
\ba\mathcal{B}^\pi = \Bigg\{x:\Omega\times [0,1]\mapsto \mathbb{R},\qquad x(t) = x(0)+ \left(x(1)-x(0) \right)t+ \sum_{m=0}^{\infty} \sum_{k=0}^{N(\pi^{m+1})-N(\pi^m)} \eta_{m,k} e^{\pi}_{m,k}(t) \nonumber\qquad\qquad\\ \text{ where, }\; \mathbb{E}(\eta_{m,k}) = 0,\; \mathbb{E}(\eta_{m,k}\eta_{m',k'}) = \delta_{m,m'}\delta_{k,k'},\; \mathbb{E}(\eta_{m,k}^4)\leq M<\infty,\; \text{ and, }\qquad\qquad\\\text{for integers }\alpha+\beta+\gamma+\delta=4, \;\mathbb{E}(\eta_{m,k}^\alpha\eta_{m_1,k_1}^\beta\eta_{m_2,k_2}^\gamma\eta_{m_3,k_3}^\delta) =\mathbb{E}(\eta_{m,k}^\alpha)\mathbb{E}(\eta_{m_1,k_1}^\beta)\mathbb{E}(\eta_{m_2,k_2}^\gamma)\mathbb{E}(\eta_{m_3,k_3}^\delta)  \Bigg\}.\quad\nonumber\ea
Then for any $x\in \mathcal{B}^\pi $, we have $[x]_{\pi^n}(t) \to t $ in probability.
Furthermore, if $\pi$ is also balanced and complete refining partition sequence, then the convergence is  almost surely.
\begin{corollary}
For any balanced complete refining sequence of partitions $\pi$, we have $\mathcal{B}^\pi \subset Q_\pi([0,1],\mathbb{R})$ almost surely.
\end{corollary}
\subsection{Processes with prescribed quadratic variation}

A well known method for constructing a process with prescribed quadratic variation is via time-changed Brownian motion.
Let $W$ be a Wiener process on a probability space $(\Omega, {\cal F},\mathbb{P})$. Then for any continuous increasing function $\phi:[0,\infty)\to [0,\infty)$   with $\phi(0)=0$ the process
$Y(t)=W(\phi(t))$ and any refining partition $\pi$, by L\'evy's theorem we have
  $$[Y]_{\pi}(t)=\phi(t)$$ almost surely. 
\par In this subsection, we will construct a class of processes with this property, using a different construction based on the Schauder expansion. We will show that our class contains time-changed Brownian motion, but also other processes which may not be semimartingales.
\par Without loss of generality for the rest of this section we will also assume  $\phi(1)=1$. We first study the Schauder expansion of a time-changed  Brownian motion: the proof of the following is based on straightforward calculations.
\begin{lemma}[Schauder expansion of a time-changed Brownian motion]\label{Brownian.coeff.timechange}
Let $\pi$ to be a finitely refining sequence of partitions and  $Y(t)=W(\phi(t))$, where $W$ is a   Brownian motion and $\phi:[0,\infty)\to [0,\infty)$ an increasing function with $\phi(0)=0$.
Then $Y$ has the following Schauder expansion:
\[Y(t) =Y(0)+ (Y(1)-Y(0))t + \sum_{m=0}^{\infty}\sum_{k=1}^{N(\pi^{m+1})-N(\pi^m)}\eta_{mk}(Y)e^{\pi}_{m,k}(t), \]
where $\eta_{m,k}(Y)\sim \mathcal{N}(0,w^{\pi,\phi}_{m,k})$  are independent and 
\begin{equation}
     w^{\pi,\phi}_{m,k}= \frac{(\phi(t^{m,k}_2)-\phi(t^{m,k}_1))(t^{m,k}_3-t^{m,k}_2)^2+(t^{m,k}_2-t^{m,k}_1)^2(\phi(t^{m,k}_3)-\phi(t^{m,k}_2))}{(t^{m,k}_2-t^{m,k}_1)(t^{m,k}_3-t^{m,k}_2)(t^{m,k}_3-t^{m,k}_1)}, \label{eq.w}
\end{equation}
where $[t^{m,k}_1,t^{m,k}_3]={\rm supp}(e_{m,k}^\pi)$  and $e_{m,k}^\pi$ attains its maximum at   $t^{m,k}_2$.
\end{lemma}

We note that   $w^{\pi,\phi}_{m,k}$ are non-random and only depend on the partition sequence and the function $\phi$.
\par For any finitely refining sequence of partitions $\pi$, and for any continuous increasing function $\phi$ with $\phi(0)=0$, similar to Equation \eqref{identity1},\ref{identity2},\ref{identity3} we have the corresponding identities (which are only dependent on $\pi$ and $\phi$ but not on the path).
\ba \label{identity1.timechange} \lim_{n\to \infty}\left[\sum_{m=0}^{n-1}\sum_{k}a^{n,\pi}_{m,k}(t)w^{\pi,\phi}_{m,k} \right] = \phi(t) , \ea
\ba\label{identity2.timechange}
  \lim_{n\to \infty}\left[\sum_{m=0}^{n-1}\sum_{k}(a^{n,\pi}_{m,k}(t))^2(w^{\pi,\phi}_{m,k})^2 \right] =0,
\ea
\begin{equation}\label{identity3.timechange}
  \lim_{n\to \infty}\left[\sum_{m, m'=0}^{n-1}\sum_{\substack{k,k'\\(m,k)\neq(m',k')}}(b^{n,\pi}_{m,k,m',k'}(t))^2w^{\pi,\phi}_{m,k}w^{\pi,\phi}_{m',k'} \right] = 0.
\end{equation}
\par The following theorem provides us with a broader class of processes with prescribed quadratic variation:
\begin{theorem}\label{mainpart1.timechange}
Let $\pi$ be a finitely refining sequence of partitions  with vanishing mesh $|\pi^n| \to 0$ and    $\phi:[0,\infty)\to [0,\infty)$ an increasing function with $\phi(0)=0$. 
Define $x: $ 
\[x(t) =x(0)+ (x(1)-x(0))t + \sum_{m=0}^{\infty}\sum_{k=1}^{N(\pi^{m+1})-N(\pi^m)}\eta_{mk}e^{\pi}_{m,k}(t).\]
where $(\eta_{m,k}, m\in \mathbb{N}, k=1..N(\pi^{m+1})-N(\pi^m) )$ is a family of random variables with  $$\mathbb{E}\eta_{m,k} = 0, \qquad \mathbb{E}\eta_{m,k}\eta_{m',k'} = \mathbbm{1}_{m,m'}\mathbbm{1}_{k,k'} w^{\pi,\phi}_{m,k},\qquad \mathbb{E}\eta_{m,k}^4<\infty$$
where $w^{\pi,\phi}_{m,k}$ is given by \eqref{eq.w} and  $$\mathbb{E}(\eta_{m,k}^\alpha\eta_{m_1,k_1}^\beta\eta_{m_2,k_2}^\gamma\eta_{m_3,k_3}^\delta) =\mathbb{E}(\eta_{m,k}^\alpha)\mathbb{E}(\eta_{m_1,k_1}^\beta)\mathbb{E}(\eta_{m_2,k_2}^\gamma)\mathbb{E}(\eta_{m_3,k_3}^\delta)$$  for all integers $\alpha,\beta,\gamma,\delta$ such that $\alpha+\beta+\gamma+\delta=4$.
Then
\[\forall\; \epsilon>0; \qquad \lim_{n\to \infty}\mathbb{P}(|[x]_{\pi^n}(t)-\phi(t)|>\epsilon)=0.\]
Furthermore, if the sequence of partitions $\pi$ is complete refining and balanced and   $\phi$ has a bounded derivative then
\[x\in Q_\pi([0,1],\mathbb{R}) \text{ almost surely } \qquad \text{ and } \qquad \mathbb{P}\left([x]_\pi(t)= \phi(t)\right)=1.\]
\end{theorem}
\begin{proof}
We skip the proof of the above theorem as the proof is very similar to the proof of Theorem \ref{mainpart1}.  The proof in particular uses Identity \eqref{identity1.timechange}, \eqref{identity2.timechange}, \eqref{identity3.timechange}.
For the proof of almost sure convergence, we use the fact that if $\phi$ has bounded derivatives and $\pi$ is balanced, then the weights $ w^{\pi,\phi}_{m,k}$ are almost surely bounded as well.
\end{proof}
The assumptions of $\pi$ and $\phi$ for almost sure convergence in Theorem \ref{mainpart1.timechange} are sufficient conditions but not necessary. To summarise, for any finitely refining sequence of partitions $\pi$ and for any continuous increasing function $\phi$ with $\phi(0)=0$, define the class of processes $\mathcal{B}_0^{\pi,\phi}$ as follows.
\ba\mathcal{B}_0^{\pi,\phi} = \Bigg\{x:\Omega\times [0,1]
\mapsto \mathbb{R}:\; x(t) = x(0)+ \left(x(1)-x(0) \right)t+ \sum_{m=0}^{\infty} \sum_{k=0}^{N(\pi^{m+1})-N(\pi^m)} \eta_{m,k} e^{\pi}_{m,k}(t) \qquad \qquad \nonumber\\ \text{ with, }\; \mathbb{E}(\eta_{m,k}) = 0,\; \mathbb{E}(\eta_{m,k}\eta_{m',k'}) = \delta_{m,m'}\delta_{k,k'}w_{m,k},\; \mathbb{E}(\eta_{m,k}^4)\leq M<\infty,\; \text{ and, }\qquad\qquad \\\mathbb{E}(\eta_{m,k}^\alpha\eta_{m_1,k_1}^\beta\eta_{m_2,k_2}^\gamma\eta_{m_3,k_3}^\delta) =\mathbb{E}(\eta_{m,k}^\alpha)\mathbb{E}(\eta_{m_1,k_1}^\beta)\mathbb{E}(\eta_{m_2,k_2}^\gamma)\mathbb{E}(\eta_{m_3,k_3}^\delta) \text{ whenever int. }\alpha+\beta+\gamma+\delta=4  \Bigg\}.\nonumber\ea
Then for any $x\in \mathcal{B}_0^{\pi,\phi} $, we have $[x]_{\pi^n}(t) \to \phi(t) $ in probability. If $\pi$ is also balanced, complete refining and the continuous increasing function $\phi$ has $\phi(0)=0$ and bounded derivatives then the convergence is in an almost sure sense.
\begin{corollary}
Let $\pi$ be any finitely refining sequence of partition and $\phi\in C^0([0,1],\mathbb{R})$ be an increasing function with $\phi(0)=0$. Then the time changed Brownian motion defined as $Y(t)=W(\phi(t))$ belongs to the class $\mathcal{B}_0^{\pi,\phi}$
\end{corollary}
\begin{corollary}
For any balanced complete refining sequence of partitions $\pi$ and for any increasing $\phi\in C^0([0,1],\mathbb{R})$  with bounded derivatives, we have $\mathcal{B}_0^{\pi,\phi} \subset Q_\pi([0,1],\mathbb{R})$ almost surely.
\end{corollary}

\section{A class of processes with quadratic variation invariant under coarsening}\label{sec5}
The quadratic variation of a path along a sequence of partitions strongly depends on the chosen sequence of partitions. As shown by Freedman \cite[p. 47]{freedman}, given any continuous function, one can always construct a sequence of partitions along which the  quadratic variation is  zero.
This result has been extended by Davis et al. \cite{davis2018} where they have shown that, 
given any continuous path $x\in C^0([0,T],\mathbb{R})$ and any increasing function $A:[0,T]\to\mathbb{R}_+$ (not necessarily continuous) one can construct a partition sequence $\pi$ such that $ [x]_{\pi}=A$. Another result by Schied \cite{schied2016b} provides a way to construct a vector space of functions with a prescribed quadratic variation. Notwithstanding these negative results, the quadratic variation of a function along a sequence of partitions $\pi$ is always the same as that along any subsequences of $\pi$ and the recent paper \cite{das2020} also identifies a class of partitions and a class of $d$-dimensional paths where  quadratic variation is partition invariant. In this section, we shall identify a class of processes $x$ for which $[x]_\pi$ is uniquely defined across any coarsening of the initial finitely refining partition $\pi$. 

One main difficulty in comparing the quadratic variation along two different partition sequences is the lack of structural similarity between the two sequences of partitions and/or lack of local bounds on the number of partition intervals.

For Brownian motion almost surely for any refining sequence of partitions $\pi$ the quadratic variation is linear and same across partitions, i.e. $\mathbb{P}([W]_\pi(t)=t)=1$. Now from Lemma \ref{Brownian.coeff} we can see along any finitely refining partition sequence $\pi$ the coefficients $\eta^\pi_{m,k}$ for Brownian motion are IID $\mathcal{N}(0,1)$. So for two finitely refining partition sequences $\pi$ and $\sigma$, if we compare the corresponding Schauder basis coefficients $\eta^\pi_{m,k}$ and $\eta^\sigma_{m,k}$ for Brownian motions, both of them have the same distribution. This uniformity of coefficients of Brownian motion contributes towards partition sequence independent quadratic variation of Brownian motions.  
\par In this section, we provide a class of `rough' continuous processes for which the Schauder expansion has similar properties across certain `related' sequences of refining partitions. As expected, our `rough' class contains Brownian motion but also contains processes that are smoother than Brownian motion in terms of H\"older continuity.  

\subsection{Invariance of quadratic variation}
\paragraph{Coarsening} A partition may be refined by adding points to it. The inverse operation, which we  call coarsening, corresponds  to  removing  points  i.e. subsampling or grouping of partition points. We will be specifically interested in coarsening that preserve the finitely refining property but may modify the asymptotic rate of decrease of the mesh size:

\begin{definition}[Coarsening of a partition sequence]\label{def.coarsening}
Let $\pi^n=(0=t^{n}_0<t^n_1< \cdots<t^{n}_{N(\pi^n)}=T)$ be a finitely refining sequence of partitions of $[0,T]$ with vanishing mesh $|\pi^n|\to 0$. A coarsening of $\pi$ is a sequence of subpartitions of $\pi^n$:
$$ A^n=(0=t^n_{p(n,0)}<t^n_{p(n,1)}<\cdots <t^n_{p(n,N(A^n))}=T), $$
 such that $(A^n)_{n\geq 1}$ is a finitely refining partition sequence of $[0,T]$.
\end{definition}
\begin{remark}
$t\in A^n$ implies $t\in \pi^n$. Also if $\sigma=(\sigma^n)_{n\geq 1}$ is a coarsening of $\pi=(\pi^n)_{n\geq 1}$, then  for any subsequence $\tau=(\pi^{K(n)})_{n\geq 1}$ of $\pi$; $\sigma^{K(n)}$ is also a coarsening of $\tau$.
\end{remark}  
\par Take $\pi$ be a finitely refining sequence of partitions of $[0,1]$ and take $\sigma=(\sigma^n)_{n\geq 1}$ to be a coarsening of $\pi$. Let $x\in C^0([0,1],\mathbb{R})$. 
Then the $x$ can be expanded along the non-uniform  Schauder system corresponding to partition sequences $\pi$ and $\sigma$ respectively i.e.
\[x(t) =x(0)+(x(1)-x(0))t+ \sum_{m=0}^{\infty}\sum_{k=1}^{N(\pi^{m+1})-N(\pi^m)}\eta_{m,k}e^{\pi}_{m,k}(t) \]
\[=x(0)+(x(1)-x(0))t+\sum_{j=0}^{\infty}\sum_{l=1}^{N(\sigma^{j+1})-N(\sigma^j)}\theta_{j,l}e^{\sigma}_{j,l}(t),\]
where, $\{\eta_{m,k}\} $ and $\{\theta_{j,l}\}$ are corresponding coefficients of the Schauder system expansion along sequence of partition $\pi$ and $\sigma$ respectively. If the support of the function $e_{j,l}^\sigma$ is $[s^{j,l}_1,s^{j,l}_3]$ and its maximum is attained at time $s^{j,l}_2$ then, the coefficient $\theta_{j,l}$ has a closed form representation as follows (Proposition \ref{coeff_hat_func}):
\[\theta_{j,l} = \frac{\bigg[\left(x(s^{j,l}_{2})-x(s^{j,l}_{1})\right)(s^{j,l}_{3}-s^{j,l}_{2})-\left(x(s^{j,l}_{3})- x(s^{j,l}_{2})\right)(s^{j,l}_{2}-s^{j,l}_1) \bigg]}{\sqrt{(s^{j,l}_2-s^{j,l}_1)(s^{j,l}_3-s^{j,l}_2)(s^{j,l}_3-s^{j,l}_1)}} \]
\[=\Bigg[\frac{(s^{j,l}_{3}-s^{j,l}_{2})\left(\sum_{m=0}^{\infty}\sum_{k=1}^{N(\pi^{m+1})-N(\pi^m)}\eta_{m,k}\left(e^{\pi}_{m,k}(s^{j,l}_2)-e^{\pi}_{m,k}(s^{j,l}_1)\right)\right)}{\sqrt{(s^{j,l}_2-s^{j,l}_1)(s^{j,l}_3-s^{j,l}_2)(s^{j,l}_3-s^{j,l}_1)}}
\]
\[-\frac{(s^{j,l}_{2}-s^{j,l}_1)\left(\sum_{m=0}^{\infty}\sum_{k=1}^{N(\pi^{m+1})-N(\pi^m)}\eta_{m,k}\left(e^{\pi}_{m,k}(s^{j,l}_3)-e^{\pi}_{m,k}(s^{j,l}_2)\right)\right) }{\sqrt{(s^{j,l}_2-s^{j,l}_1)(s^{j,l}_3-s^{j,l}_2)(s^{j,l}_3-s^{j,l}_1)}}\Bigg]\]
\[=\sum_{m=0}^{\infty}\sum_{k=1}^{N(\pi^{m+1})-N(\pi^m)} \left[\frac{(s^{j,l}_{3}-s^{j,l}_{2})\left(e^{\pi}_{m,k}(s^{j,l}_2)-e^{\pi}_{m,k}(s^{j,l}_1)\right)-(s^{j,l}_{2}-s^{j,l}_{1})\left(e^{\pi}_{m,k}(s^{j,l}_3)-e^{\pi}_{m,k}(s^{j,l}_2)\right)}{\sqrt{(s^{j,l}_2-s^{j,l}_1)(s^{j,l}_3-s^{j,l}_2)(s^{j,l}_3-s^{j,l}_1)}} \right]\eta_{m,k}. \]
Denote \begin{equation}
 A^{m,k}_{j,l} =\frac{(s^{j,l}_{3}-s^{j,l}_{2})\left(e^{\pi}_{m,k}(s^{j,l}_2)-e^{\pi}_{m,k}(s^{j,l}_1)\right)-(s^{j,l}_{2}-s^{j,l}_{1})\left(e^{\pi}_{m,k}(s^{j,l}_3)-e^{\pi}_{m,k}(s^{j,l}_2)\right)}{\sqrt{(s^{j,l}_2-s^{j,l}_1)(s^{j,l}_3-s^{j,l}_2)(s^{j,l}_3-s^{j,l}_1)}}.\label{eq.A}\end{equation}
Since the function $e^\pi_{m,k}$ only depends on $\pi$ not on the path $x\in C^0([0,1],\mathbb{R})$, the coefficient $A^{m,k}_{j,l}$ only depends on the refining partitions $\sigma$ and $\pi$ but not on the continuous path $x$. So the expression for $\theta_{j,l}$ can be represented as an infinite expansion of $\eta$'s.
\begin{equation}\label{coff}
  \theta_{j,l} =\sum_{m=0}^{\infty}\sum_{k=1}^{N(\pi^{m+1})-N(\pi^m)}A^{m,k}_{j,l}\eta_{m,k}.
\end{equation}
The above equation holds for any two finitely refining partitions, but since $\sigma$ is a coarsening of $\pi$, $A^{m,k}_{j,l}=0$ for all $m>j+1$, $\forall l,k$. So the Equation \eqref{coff} reduces to:
\begin{equation}\label{coff2}
  \theta_{j,l} =\sum_{m=0}^{j+1}\sum_{k=1}^{N(\pi^{m+1})-N(\pi^m)}A^{m,k}_{j,l}\eta_{m,k}.
\end{equation}
Now if we take the path $x$ to be a typical path of Brownian motion, then $\eta_{m,k}\sim^{IID}\mathcal{N}(0,1)$ and
$\theta_{j,l}\sim^{IID}\mathcal{N}(0,1)$. So,
\[\mathbb{E}\theta^2_{j,l} = 1\]
\[\implies\mathbb{E}\left[\sum_{m=0}^{j+1}\sum_{k=1}^{N(\pi^{m+1})-N(\pi^m)}A^{m,k}_{j,l}\eta_{m,k}\right]^2= 1.\]
For Brownian motion $\mathbb{E}\eta_{m,k}\eta_{m',k'} = \mathbb{E}\delta_{m,m'}\delta_{k,k'} = \mathbbm{1}_{m=m'}\mathbbm{1}_{k=k'}$ and for any fixed pair $(j,l)$, the above sum is a finite sum. So the above equality reduces to: 
\[\left[\sum_{m=0}^{j+1}\sum_{k=1}^{N(\pi^{m+1})-N(\pi^m)}\sum_{m'=0}^{j+1}\sum_{k'=1}^{N(\pi^{m'+1})-N(\pi^{m'})}(A^{m,k}_{j,l})(A^{m',k'}_{j,l})\mathbb{E}\left(\eta_{m,k}\eta_{m',k'}\right)\right]= 1\]
\[\implies \sum_{m=0}^{j+1}\sum_{k=1}^{N(\pi^{m+1})-N(\pi^m)}(A^{m,k}_{j,l})^2\mathbb{E}\eta^2_{m,k}= 1\]
\begin{equation}\label{eq1}\implies\sum_{m=0}^{j+1}\sum_{k=1}^{N(\pi^{m+1})-N(\pi^m)}(A^{m,k}_{j,l})^2=1. \end{equation}
Similarly, for Brownian motion the cross-correlations of the coefficients are $0$. So for pairs $(j,l)\neq(j',l')$:
\[\mathbb{E}(\theta_{j,l}\theta_{j',l'} )= 0\]
\[\implies E\left[\left(\sum_{m=0}^{j+1}\sum_{k=1}^{N(\pi^{m+1})-N(\pi^m)}A^{m,k}_{j,l}\eta_{m,k}\right) \left(\sum_{m'=0}^{j'+1}\sum_{k'=1}^{N(\pi^{m'+1})-N(\pi^{m'})}A^{m',k'}_{j',l'}\eta_{m',k'}\right) \right] = 0\]
\[\implies\left[\sum_{m=0}^{(j\wedge j') +1}\sum_{k=1}^{N(\pi^{m+1})-N(\pi^m)}A^{m,k}_{j,l}A^{m,k}_{j',l'}\left(\mathbb{E}\eta_{m,k}^2\right) \right] = 0\]
\begin{equation}\label{eq2}
\implies\sum_{m=0}^{(j\wedge j') +1}\sum_{k=1}^{N(\pi^{m+1})-N(\pi^m)}A^{m,k}_{j,l}A^{m,k}_{j',l'} = 0 . 
\end{equation}
Comparing the fourth moment of the coefficient $\theta_{j,l}$ for Brownian paths we get:
\[\mathbb{E}\theta_{j,l}^4 = 3\left(\mathbb{E}\theta_{j,l}^2\right)^2\]
\[\implies \mathbb{E}\left(\sum_{m=0}^{j+1}\sum_{k=1}^{N(\pi^{m+1})-N(\pi^m)}A^{m,k}_{j,l}\eta_{m,k}\right)^4 = 3\]
\[\implies \mathbb{E}\left(\sum_{m=0}^{j+1}\sum_{k=1}^{N(\pi^{m+1})-N(\pi^m)}(A^{m,k}_{j,l})^4\eta^4_{m,k} + \sum_{m,k} \sum_{\substack{m','k\\ (m,k)\neq (m',k')}}(A^{m,k}_{j,l})^2(A^{m',k'}_{j,l})^2\eta^2_{m,k}\eta^2_{m',k'}\right) = 3\]
\[\implies \left(\sum_{m=0}^{j+1}\sum_{k=1}^{N(\pi^{m+1})-N(\pi^m)}3(A^{m,k}_{j,l})^4 + \sum_{m,k} \sum_{\substack{m','k\\ (m,k)\neq (m',k')}}(A^{m,k}_{j,l})^2(A^{m',k'}_{j,l})^2\right) = 3\]
\[\implies \left[\sum_{m=0}^{j+1}\sum_{k=1}^{N(\pi^{m+1})-N(\pi^m)}2(A^{m,k}_{j,l})^4 +\left( \sum_{m,k}(A^{m,k}_{j,l})^2\right)\left( \sum_{m','k}(A^{m',k'}_{j,l})^2\right)\right] = 3.\]
Substituting Equation \eqref{eq1} twice in the second sum we get the following identity:
\begin{equation}\label{eq3}
  \sum_{m=0}^{j+1}\sum_{k=1}^{N(\pi^{m+1})-N(\pi^m)}(A^{m,k}_{j,l})^4 =1.
\end{equation}
Similarly, exploring the uncorrelated property of the coefficients $\theta$ for Brownian motion leads to the following equalities:
\begin{equation}\label{eq4}
   \sum_{m=0}^{(j\wedge j') +1}\sum_{k=1}^{N(\pi^{m+1})-N(\pi^m)}\left((A^{m,k}_{j,l})^2(A^{m,k}_{j',l'})^2\right) =0 \quad\text{ and,}
\end{equation}
\begin{equation}\label{eq6}
   \sum_{m=0}^{(j\wedge j') +1}\sum_{k=1}^{N(\pi^{m+1})-N(\pi^m)}\left((A^{m,k}_{j,l})^3(A^{m,k}_{j',l'})\right) =0 \quad\text{ and,}
\end{equation}
\begin{equation}\label{eq7}
   \sum_{m=0}^{j\wedge j' \wedge j_1 +1}\sum_{k=1}^{N(\pi^{m+1})-N(\pi^m)}\left((A^{m,k}_{j,l})^2(A^{m,k}_{j',l'})(A^{m,k}_{j_1,l_1})\right) =0 \quad\text{ and,}
\end{equation}
\begin{equation}\label{eq8}
   \sum_{m=0}^{j \wedge j' \wedge j_1\wedge j_2 +1}\sum_{k=1}^{N(\pi^{m+1})-N(\pi^m)}\left((A^{m,k}_{j,l})(A^{m,k}_{j',l'})(A^{m,k}_{j_1,l_1})(A^{m,k}_{j_2,l_2})\right) =0 .
\end{equation}

The following theorem provides properties of  Schauder coefficients represented along two different partition sequences which are coarsening of each other. 

\begin{theorem}\label{mainpart2}
Let $\pi$ be a finitely refining sequence of partitions of $[0,1]$ and  $\sigma=(\sigma^n)_{n\geq 1}$ be  a coarsening of $\pi$. Define for $t\in [0,1]$
\[x(t) =x(0)+(x(1)-x(0))t+ \sum_{m=0}^{\infty}\sum_{k=1}^{N(\pi^{m+1})-N(\pi^m)}\eta_{m,k}e^{\pi}_{m,k}(t)\]
where \begin{equation}
\mathbb{E}\eta_{m,k} = 0,\; \mathbb{E}\eta_{m,k}\eta_{m',k'} = \mathbbm{1}_{m,m'} \mathbbm{1}_{k,k'},\qquad \mathbb{E}\eta_{m,k}^4=M<\infty\qquad {\rm and}
    \label{eq.property1}
\end{equation}
\begin{equation}\mathbb{E}(\eta_{m,k}^\alpha\eta_{m_1,k_1}^\beta\eta_{m_2,k_2}^\gamma\eta_{m_3,k_3}^\delta) =\mathbb{E}(\eta_{m,k}^\alpha)\mathbb{E}(\eta_{m_1,k_1}^\beta)\mathbb{E}(\eta_{m_2,k_2}^\gamma)\mathbb{E}(\eta_{m_3,k_3}^\delta)\label{eq.property2}
\end{equation} for all integer exponents $\alpha,\beta,\gamma,\delta$ satisfying $\alpha+\beta+\gamma+\delta=4$.
Then $(\theta_{j,l}, j\in \mathbb{N}, 1\leq l\leq N(\sigma^{l+1})-N(\sigma^l) )$ defined by Equations \eqref{eq.A}-\eqref{coff2} also satisfies the properties \eqref{eq.property1}-\eqref{eq.property2}.
\end{theorem}
\begin{proof}
$\mathbb{E}\theta_{j,l}$ and $\mathbb{E}\theta_{j,l}^2$ can be expanded as follows. 
\[\mathbb{E}\theta_{j,l} =\mathbb{E}\left[ \sum_{m=0}^{j+1}\sum_{k=1}^{N(\pi^{m+1})-N(\pi^m)}A^{m,k}_{j,l}\eta_{m,k}\right]=\sum_{m=0}^{j+1}\sum_{k=1}^{N(\pi^{m+1})-N(\pi^m)}A^{m,k}_{j,l}\mathbb{E}\eta_{m,k} =0 \; \text{ and,} \]
\[\mathbb{E}\theta_{j,l}^2 =\mathbb{E} \left(\sum_{m=0}^{j+1}\sum_{k=1}^{N(\pi^{m+1})-N(\pi^m)}A^{m,k}_{j,l}\eta_{m,k}\right)^2 = \sum_{m=0}^{j+1}\sum_{k=1}^{N(\pi^{m+1})-N(\pi^m)}(A^{m,k}_{j,l})^2= 1. \]
The last identity follows from the Equation \eqref{eq1}. For the covariation the following identity can be obtained.
\[\mathbb{E}\theta_{j,l}\theta_{j',l'} =\mathbb{E} \left[ \left(\sum_{m=0}^{j+1}\sum_{k=1}^{N(\pi^{m+1})-N(\pi^m)}A^{m,k}_{j,l}\eta_{m,k}\right)\left(\sum_{m'=0}^{j'+1}\sum_{k'=1}^{N(\pi^{m'+1})-N(\pi^{m'})}A^{m',k'}_{j',l'}\eta_{m',k'}\right)\right] \]
\[= \sum_{m=0}^{j\wedge j' +1}\sum_{k=1}^{N(\pi^{m+1})-N(\pi^m)}(A^{m,k}_{j,l})(A^{m,k}_{j',l'})= 0. \]
The last equality follows from Equation \ref{eq2}. Now the fourth moment of $\theta_{j,l}$ can be represented as follows.
\[\mathbb{E}\theta_{j,l}^4 =\mathbb{E} \left(\sum_{m=0}^{j+1}\sum_{k=1}^{N(\pi^{m+1})-N(\pi^m)}A^{m,k}_{j,l}\eta_{m,k}\right)^4\]
\[ = \sum_{m=0}^{j+1}\sum_{k=1}^{N(\pi^{m+1})-N(\pi^m)}(A^{m,k}_{j,l})^4\mathbb{E}\eta_{m,k}^4 +\sum_{m,k}\sum_{\substack{m',k'\\(m,k)\neq(m',k')}} \left(A^{m,k}_{j,l}\right)^2\left(A^{m',k'}_{j,l}\right)^2 \mathbb{E}(\eta_{m,k}^2\eta_{m',k'}^2)\]
\[ \leq M\sum_{m=0}^{j+1}\sum_{k=1}^{N(\pi^{m+1})-N(\pi^m)}\left(A^{m,k}_{j,l}\right)^4 +\sum_{m,k}\sum_{\substack{m',k'\\(m,k)\neq(m',k')}} \left(A^{m,k}_{j,l}\right)^2\left(A^{m',k'}_{j,l}\right)^2 \]
\[ = (M-1)\sum_{m=0}^{j+1}\sum_{k=1}^{N(\pi^{m+1})-N(\pi^m)}\left(A^{m,k}_{j,l}\right)^4 +\left(\sum_{m,k}\left(A^{m,k}_{j,l}\right)^2 \right)\left(\sum_{m',k'}\left(A^{m',k'}_{j,l}\right)^2\right) <\infty.\]
The last inequality follows from the fact $\mathbb{E}\eta_{m,k}^4 = M$ and using Equation \eqref{eq1} and \eqref{eq3}. The uncorrelated property of $\theta$ is a consequence of Equation (\ref{eq4}, \ref{eq6}, \ref{eq7}, \ref{eq8}) and the fact that $\mathbb{E}\eta_{m,k}^4=M<\infty$. So the result follows.
\end{proof}

\begin{remark}
The assumptions of the above theorem are sufficient but may not be necessary. Note that, unlike the Brownian motion case, the coefficients in the non-uniform  Schauder basis expansion of typical paths satisfying the assumption of Theorem \ref{mainpart2} only have uncorrelated properties and do not necessarily have IID properties.
\end{remark}
For any finitely refining sequence of partitions $\pi$ of $[0,1]$  we can define the following class of processes:
\ba\mathcal{A}^\pi = \Bigg\{x:\Omega\times [0,1]
\mapsto \mathbb{R}, \quad x(t) = x(0)+ \left(x(1)-x(0) \right)t+ \sum_{m=0}^{\infty} \sum_{k=0}^{N(\pi^{m+1})-N(\pi^m)} \eta_{m,k} e^{\pi}_{m,k}(t) \qquad \quad \nonumber\\ \text{ where }\; \mathbb{E}(\eta_{m,k}) = 0,\; \mathbb{E}(\eta_{m,k}\eta_{m',k'}) = \delta_{m,m'}\delta_{k,k'},\; \mathbb{E}(\eta_{m,k}^4)=M<\infty,\; \text{ and} \qquad \qquad
\\\mathbb{E}(\eta_{m,k}^\alpha\eta_{m_1,k_1}^\beta\eta_{m_2,k_2}^\gamma\eta_{m_3,k_3}^\delta) =\mathbb{E}(\eta_{m,k}^\alpha)\mathbb{E}(\eta_{m_1,k_1}^\beta)\mathbb{E}(\eta_{m_2,k_2}^\gamma)\mathbb{E}(\eta_{m_3,k_3}^\delta) \text{ for all integers }\alpha+\beta+\gamma+\delta=4  \Bigg\}.\nonumber\ea
Then $\mathcal{A}^\pi \subset \mathcal{B}^\pi$ and we have the following result:
\begin{theorem}[Invariance of Quadratic variation]\label{main.theorem}
For any finitely refining sequence of partitions $\pi$, take a process $x\in \mathcal{A}^\pi$. Then for any coarsening $\sigma$ of $\pi$ we have:
\[\forall t\in[0,1],\; [x]_{\sigma^n}(t)\to t \text{ and, } [x]_{\pi^n}(t)\to t \;\text{ in probability} . \]
Furthermore, if both $\pi$ and $\sigma$ are complete refining and balanced then:
\[\mathbb{P}\left(x\in Q_\pi([0,1],\mathbb{R}) \cap Q_\sigma([0,1],\mathbb{R})\right)=1 \; \text{and, } [x]_\pi(t)=[x]_\sigma(t) \;\text{ almost surely.} \]
\end{theorem}

\begin{proof}
Since $x\in \mathcal{A}^\pi$ for a finitely refining sequence of partitions $\pi$ of $[0,1]$, $x\in C^0([0,1],\mathbb{R})$. Now for any coarsening $\sigma$ of $\pi$, Theorem \ref{mainpart2} concludes the corresponding Schauder coefficients $\theta^\sigma_{j,l}$ and $\eta^\pi_{m,k}$ have same uncorrelated properties. So the result follows from Theorem \ref{mainpart1}. 
\end{proof}
\begin{figure}[ht!]
    \centering
    \includegraphics[width=15cm]{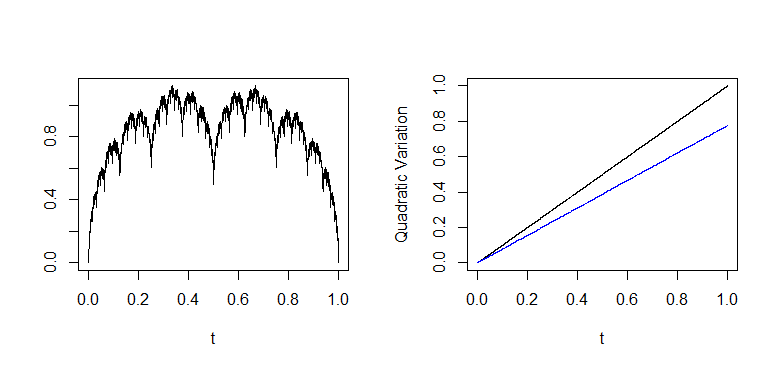}
    \caption{Left: Plot of the function $x$ defined in Example \ref{ex-schied-diff-qv} truncated at $n=12$. Right: The black line represented the quadratic variation of the function $x$ at level n=12 with respect to dyadic partition. The blue line represents the quadratic variation of the function $x$ at level n=12 with respect to the partition $\pi$.}
    \label{fig-1}
    \includegraphics[width=15cm]{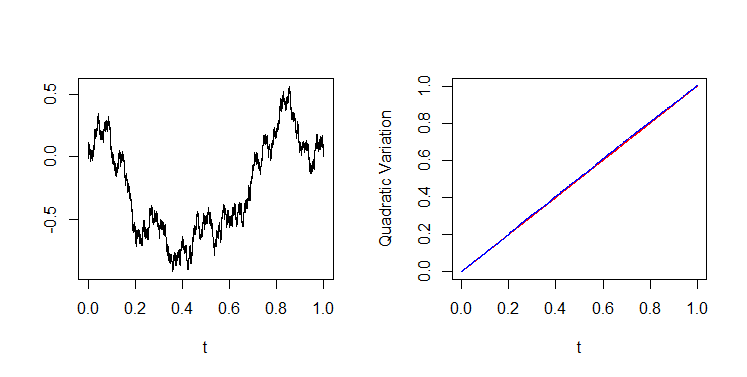}
    \caption{Left: Sample path of $X$ defined in Example \ref{ex-schied-same-qv} truncated at $n=12$. Right: The red line represented the quadratic variation of  $X$at level n=12 along the dyadic partition. The blue line represents the quadratic variation of   $X$ at level n=12 along $\pi$.}
    \label{fig-2}
\end{figure}
The following is an example of a path that does not satisfy the assumptions of Theorem \ref{mainpart2} and whose quadratic variation (unlike Theorem \ref{main.theorem}) is not invariant  under coarsening. 
\begin{example}[Example of continuous function with different quadratic variation along two different balanced finitely refining sequence of partition]\label{ex-schied-diff-qv}
Define
\[x(t)= \sum_{n=0}^{\infty} \sum_{k=0}^{2^m-1} e^{\mathbb{T}}_{m,k}(t),\]
Then the quadratic variation of $x$ along $\mathbb{T}$ is different from the quadratic variation of $x$ along $\pi$, where $\pi^n=\left(0,\frac{1}{2^n},\frac{2}{2^n},\frac{4}{2^n}\cdots\frac{3i+1}{2^n},\frac{3i+2}{2^n},\cdots,1\right)$. Note that the function $x$ belongs to the class of functions defined in \cite{schied2016} and both the partition sequences $\pi$ and $\mathbb{T}$ are finitely refining and $\pi$ is coarsening of $\mathbb{T}$. Also, $x$ has linear quadratic variation along both sequence of partitions $\pi$ and $\mathbb{T}$, but they are not same for all $t\in (0,1]$.
\end{example}

Not surprisingly, Brownian motion belongs to the class $\mathcal{A}^\pi$ for any finitely refining sequence of partitions $\pi$, as for Brownian motion the coefficient of Schauder system expansion follows IID $\mathcal{N}(0,1)$. But the class of paths in $\mathcal{A}^\pi$ is not just a typical path of Brownian motion.
\begin{example}[Process with $\frac{1}{2}$-H\"older continuous paths in the class $\mathcal{A^\mathbb{T}}$ ]\label{ex-schied-same-qv}
Define
\[ X(t)= \sum_{n=0}^{\infty} \sum_{k=0}^{2^m-1} \theta_{m,k}e^{\mathbb{T}}_{m,k}(t),\]
where $\theta_{m,k}$ are IID random variables with
\[
\theta_{m,k} = 
\begin{cases}
  1, &\quad\text{with probability } \frac{1}{2} \\
  -1, &\quad \text{with probability } \frac{1}{2}
\end{cases}. 
\]
From the results of Mishura and Schied \cite{schied2016b} we know that $X\in Q_\mathbb{T}([0,1],\mathbb{R})$ and from the construction $X\in \mathcal{A}^\mathbb{T}$. The process $X$ belongs to the class $\mathcal{X}$ defined in \cite{schied2016b}, which is a class of function with $\frac{1}{2}$-H\"older continuity. So our `rough' class $\mathcal{A^\mathbb{T}}$ contains   $\frac{1}{2}$-H\"older continuous paths.
\end{example}

So $\mathcal{A}^\pi$ is an interesting class of processes and contains a processes  smoother than Brownian motion in the sense of H\"older continuity, but still `rough' enough to have  quadratic variation invariant across different finitely refining partitions.

\begin{example}\label{example5}
Let
$\pi^n= \left(0=t^n_1<\cdots<t^n_{N(\pi^n)} \right)$ where
\[\forall k=1,\cdots, 2^n, \qquad t^{n+1}_{2k} = t^n_k \; \text{ and, }\; t^{n+1}_{2k} = t^n_k+ \frac{t^n_{k+1}-t^n_k}{2.5} \]
and define $x:[0,1] \to \mathbb{R}$ as follows.
\[\forall t\in[0,1], \qquad x(t)= \sum_{n=0}^{\infty} \sum_{k=0}^{2^m-1}e^{\pi}_{m,k}(t). \]
\end{example}
\begin{example}\label{example6}
Define the sequence of partition $\pi=(\pi^n)_{n\geq 1}$ with
$\pi^n= \left(0=t^n_1<\cdots<t^n_{N(\pi^n)} \right)$ as follows.
\[\forall k=1,\cdots, 2^n, \qquad t^{n+1}_{2k} = t^n_k \; \text{ and, }\; t^{n+1}_{2k} = t^n_k+ \frac{t^n_{k+1}-t^n_k}{2.5} \]
Define the process $X:[0,1]\times \Omega \to\mathbb{R}$ as 
\[ X(t)= \sum_{n=0}^{\infty} \sum_{k=0}^{2^m-1} \theta_{m,k}e^{\pi}_{m,k}(t),\]
where $\theta_{m,k}$ are IID random variables with
\[
\theta_{m,k} =
\begin{cases}
  1, &\quad\text{with probability } \frac{1}{2} \\
  -1, &\quad \text{with probability } \frac{1}{2}
\end{cases}. 
\]
\end{example}

\begin{figure}[ht!]
    \centering
    \includegraphics[width=15cm]{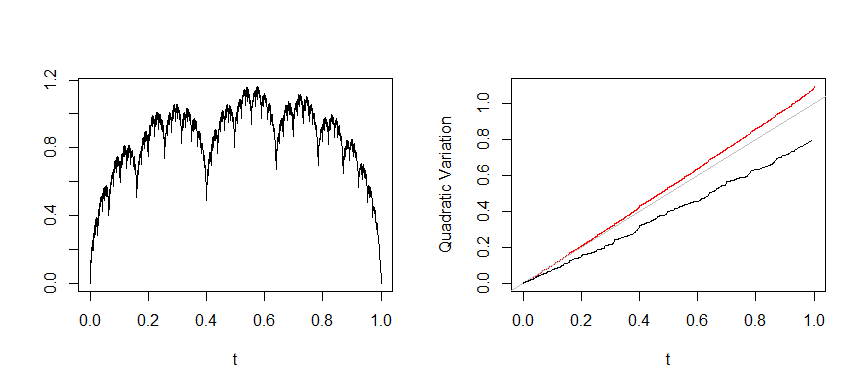}
    \caption{Left: Plot of the function $x$ defined in Example \ref{example5} truncated at level $n=12$. Right: The red line represents the quadratic variation of the function $x$ at level n=12 with respect to partition $\pi$. The black line represents the quadratic variation of the function $x$ at level n=12 for a random partition and the gray line represents for $y=x$ line.}
    \label{fig-schied-notRough-non-uniform}
    \includegraphics[width=15cm]{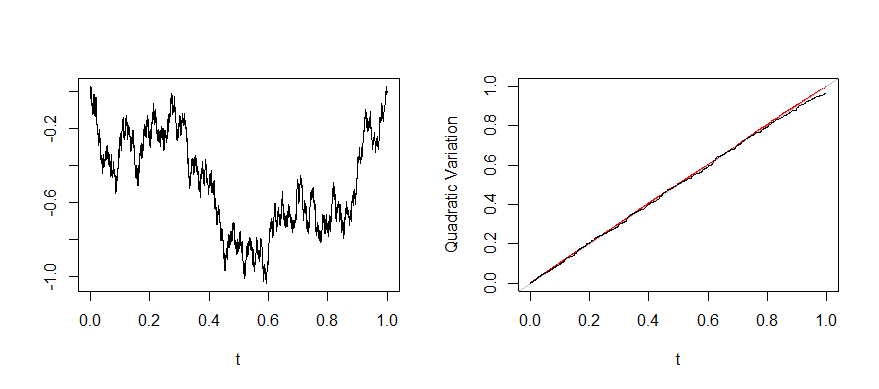}
    \caption{Left: Sample path of $X$ defined in Example \ref{example6} truncated at $n=12$. Right: The red line represented the quadratic variation of   $X$ at level n=12 with respect to partition $\pi$. The black line represents the quadratic variation of  $X$ at level n=12 for a random partition and the gray line represents for $y=x$ line.}
    \label{fig-schied-rough-non-uniform}
\end{figure}

\subsection{Properties and lemmas}

In this subsection, we will discuss some general properties of a process that contains $\mathcal{A}^\pi$, for any finitely refining sequence of partitions $\pi$.
\par For convenience of the next section let us reorder the complete orthonormal basis $\{\psi_{m,k}\}_{m,k}$ as $\{\psi_i\}_{i}$. Since $\{\psi_i\}_i$ is a set of complete orthonormal basis, for all $x\in \mathcal{A}^\pi$ we can express $x$ in the Schauder basis expansion along $\pi$ as follows.
\[x(t) = \sum_{i=0}^{\infty} \eta_{i} \int_0^t \psi^{\pi}_{i}(u)du, \]
where, $\mathbb{E}(\eta_i) = 0,\; \mathbb{E}(\eta_i\eta_{i'}) = \delta_{i,i'},\; \mathbb{E}(\eta_{i}^4)<\infty$. Now define,
\[I_t(s) := \begin{cases} 
   1 & s< t \\
   0 & s\geq t 
  \end{cases} \text{ and } <f,g> := \int_0^1 f(t)g(t)dt \]
Then, \[ \int_0^t \psi^{\pi}_{i}(u)du = <I_t,\psi_i^\pi>.\]
Since $\{\psi_i\}_i$ is a set of complete orthonormal basis we have, 
\begin{equation}\label{eq5}
  I_t = \sum_{i=0}^{\infty}<I_t,\psi_i^\pi>\psi_i^\pi \; \text{ and }\; t= ||I_t||^2 = \sum_{i=0}^{\infty}<I_t,\psi_i^\pi>^2.
\end{equation}

\begin{lemma}
For any finitely refining sequence of partitions $\pi$ take $x\in \mathcal{A}^\pi$. For any two times $t$ and $s\in [0,T]$: $\mathbb{E}[x(t)x(s)] =
t \wedge s,$ where $t \wedge s = \min(t, s)$.
\end{lemma}
\begin{proof}
Corresponding to $\pi$ we have a complete orthonormal set of basis $\{\psi^\pi_i\}_i$ (as an example non-uniform Haar basis defined in Section \ref{sec2}). So:
\[\mathbb{E}[x(t)x(s)] = \mathbb{E} \left[\left(\sum_{i=0}^{\infty} \eta_{i} \int_0^t \psi^{\pi}_{i}(u)du\right) \left(\sum_{i=0}^{\infty} \eta_{i} \int_0^s \psi^{\pi}_{i}(u)du\right)\right]\]
\[ = \sum_{i=0}^{\infty} \mathbb{E}\eta_{i}^2 \left(\int_0^t \psi^{\pi}_{i}(u)du\right)\left(\int_0^s \psi^{\pi}_{i}(u)du\right)\]
\[ = \sum_{i=0}^{\infty}< I_t, \psi^\pi_i >< I_s, \psi^\pi_i > \;
= \;< I_t, I_s >\; = t \wedge s.\]
\end{proof}
As a consequence of the above for any finitely refining sequence of partition $\pi$ and for any $x\in \mathcal{A}$, we have uncorrelated property of disjoint increments of $x$, i.e. if we have two disjoint interval $[t_1,t_2],[s_1,s_2]\subset [0,T]$ then for all $x\in \mathcal{A}^\pi$, we have $\mathbb{E}\left[(x(t_2)-x(t_1))(x(s_2)-x(s_1))\right] = 0$.

\begin{theorem}
Let $\{\phi_i\}$ be an arbitrary complete orthonormal basis and let $\eta_1,\eta_2,\eta_3\cdots$ be a sequence of random variables defined on a probability space $(\Omega, \mathcal{F}, \mathbb{P})$, with $\mathbb{E}\eta_i = 0, \mathbb{E}\eta_i^2 = 1$ and $\mathbb{E}\eta_i\eta_j = \delta_{ij}$ for $i,j = 1, 2, \cdots$, define
\begin{equation}\label{level.n}
X^n_t =\sum_{i=1}^n \eta_i \int_0^t \phi_i(s)ds. 
\end{equation}

Then for each $t$, $X^n_t$ is a Cauchy sequence in $L^2(\Omega, \mathcal{F}, \mathbb{P})$ whose limit $X_t$ is a random variable with mean zero and variance $t$.
\end{theorem}
 For any finitely refining sequence of partition $\pi$ the assumption of the above theorem is satisfied for all $x\in \mathcal{A}^\pi$.
\begin{proof}
Since $\{\phi_i\}_i$ is a complete orthonormal basis we have
\[ I_t = \sum_{i=0}^{\infty}<I_t,\phi_i>\phi_i \; \text{ and }\; t=\vert t \vert^2 = \sum_{i=0}^{\infty}<I_t,\phi_i>^2.\]
So we can have the following expression for $\mathbb{E}\left(X_t^n-X_t^m\right)^2$ where $n>m$ as follows.
\[\mathbb{E}\left(X_t^n-X_t^m\right)^2 = \mathbb{E} \left(\sum_{i=m+1}^{n} \eta_i \int_0^t\phi_i(s)ds \right)^2 =\sum_{i=m+1}^{n} \mathbb{E} \eta_i^2\left( \int_0^t\phi_i(s)ds \right)^2 \]
\[ = \sum_{i=m+1}^n <I_t,\phi_i>^2 \xrightarrow[]{m,n\to \infty} 0.\]
Thus $X_t^n$ is a Cauchy sequence in $L^2(\Omega, \mathcal{F}, \mathbb{P})$. The mean and the variance of the limiting random variable $X_t$ can be represented as:
\[\mathbb{E}X_t = \lim_{n\to \infty} \mathbb{E}X_t^n = \lim_{n\to \infty}\mathbb{E} \left(\sum_{i=0}^{n} \eta_i \int_0^t\phi_i(s)ds \right)= \lim_{n\to \infty} \sum_{i=0}^{n} \mathbb{E}\left(\eta_i\right) \int_0^t\phi_i(s)ds =0 \quad \text{ and,}\]
\[ Var( X_t) = \lim_{n\to \infty} Var (X_t^n) =\lim_{n\to \infty}\left[\mathbb{E} \left(\sum_{i=0}^{n} \eta_i \int_0^t\phi_i(s)ds \right)^2 - \left(\mathbb{E} \sum_{i=0}^{n} \eta_i \int_0^t\phi_i(s)ds \right)^2\right]\]
\[=\lim_{n\to \infty}\sum_{i=0}^{n}\mathbb{E} \left( \eta_i\right)^2 \left(\int_0^t\phi_i(s)ds \right)^2 = \lim_{n\to \infty}\sum_{i=0}^{n} <I_t,\phi_i>^2 = t.\]
So the lemma follows.
\end{proof}
The above result is valid for any orthonormal basis (non just for non-uniform Haar basis). For the following continuity result, let us assume preciously non-uniform Haar basis. So Equation \eqref{level.n} is as follows.
\begin{equation} \label{level.hat.n}
 X^n(t)= \sum_{m=0}^{n-1} \sum_{k=0}^{N(\pi^{m+1})-N(\pi^{m})-1} \theta_{m,k}e^\pi_{m,k}(t),\quad \text{ and, } 
\end{equation}
\[X(t)= \lim_{n\to \infty}x^n(t) =\sum_{m=0}^{\infty} \sum_{k=0}^{N(\pi^{m+1})-N(\pi^{m})-1} \theta_{m,k}e^\pi_{m,k}(t) \]
\begin{theorem}[Continuity of path]\label{hat.continuity}
Take a balanced, finitely refining sequence of partitions $\pi$ of $[0,1]$. Then under the assumption $\mathbb{E}(\theta_{m,k}^4)<M$, for all $m,k$ the sequence $X^n(t)$ defined in Equation \eqref{level.hat.n}, converges uniformly in $t$, almost surely to $X(t)$.
Thus the process $X(t) = \lim_{n\to \infty} X^n(t)$ is a stochastic process with continuous sample paths.
\end{theorem}
\begin{proof}
Let define $y^n(t) = X^{n+1}(t)-X^n(t) = \sum_{k=0}^{N(\pi^{n+1})-N(\pi^{n})-1} \theta_{n,k}e^\pi_{n,k}(t)$, then if we can show that the function $y^n$ is continuous and converges to $0$   uniformly so the result follows. Now since $e_{n,k}$ is a continuous function over $t$ for all $n,k$, so for every $n\in \mathbb{N}\cap \{0\}$: $y^n$ is a continuous function over $t$. Since $\pi$ is finitely refining for any fixed $n$ there exists $M< \infty$ (independent of $n$) such that at max $M$ many of ${e_{n,k}}$ are nonzero for any time $t\in [0,1]$. Now define: 
\[H_n=\sup_{t\in [0,1]}|y^n(t)|= \sup_{t\in [0,1]}|X^{n+1}(t)-X^n(t)| \]\[=\sup_{t\in [0,1]}\left| \sum_{k=0}^{N(\pi^{n+1})-N(\pi^{n})-1} \theta_{n,k}e^\pi_{n,k}(t) \right|\leq C_1|\pi^n|^{\frac{1}{2}}\times \sup_{k}|\theta_{n,k}|. \]
For the last inequality we use the fact that for a balanced sequence of partitions $\pi$, $\sup_{t\in [0,1]}|e_{n,k}(t)| \leq C|\pi^n|^{\frac{1}{2}}$.
Thus for any constant $c_n$,
\[\mathbb{P}\left(H_n > C_1|\pi^n|^{\frac{1}{2}} c_n \right) \leq \mathbb{P}\left(\sup_{k}|\theta_{n,k}| > c_n \right)= \mathbb{P}\left(\cup_k\left\{ |\theta_{n,k}|> c_n \right\}\right)\]
\begin{equation}\label{inequality.cont}
  \leq\sum_{k}\mathbb{P}\left( |\theta_{n,k}|> c_n \right)\leq C_2N(\pi^n)\times\frac{\mathbb{E}|\theta_{n,k}|^4}{c_n^4} \leq C_0\times\frac{N(\pi^n)}{c_n^4}\leq C\frac{M^n}{c_n^4},
\end{equation}
where, $C,M$ are finite constants independent of $n$. The last inequality is a consequence of Markov inequality. We now choose $c_n = |\pi^n|^{\epsilon-\frac{1}{2}}$
 for some $\epsilon>0$ with $8\epsilon<1$. Then the right-hand side of Inequality \eqref{inequality.cont} is $ C_0\frac{N(\pi^n)}{c_n^4} = C_0\frac{N(\pi^n)}{|\pi^n|^{4\epsilon-2}}\leq M_1|\pi^n|^{1- 4\epsilon} \leq M_2 M^{n(4\epsilon-1)}$ (The two inequality follows as $\pi$ is balanced). Now we know that $M^{n(4\epsilon-1)}$ is a general team of in a convergent series. Also $b_n$ defined as $b_n= C_1|\pi^n|^{\frac{1}{2}} c_n = C_1|\pi^n|^{\frac{1}{2}} |\pi^n|^{\epsilon-\frac{1}{2}} = C_1 |\pi^n|^\epsilon \to 0 $ as $n\to \infty$. 
 So using Borel-Cantelli Lemma, Inequality \eqref{inequality.cont} deduces to, 
\[\mathbb{P}[H_n > b_n \text{ infinitely often }] = 0\]
Since $b_n \to 0$, this shows that $H_n$ is a convergent series and completes the proof.
\end{proof}
In Theorem \ref{hat.continuity}, the assumption of the reference partition $\pi$ to be balanced is sufficient but \textit{not} necessary. 

\section{Extension to the multidimensional case}\label{sec6}
In this section, we extend the previous results discussed to a  multidimensional setting. 
\paragraph{Non-uniform multidimensional Haar basis.} Fix a finitely refining sequence of partitions $\pi$ of $[0,1]$. The one dimensional non-uniform Haar basis can be represented as $\{h_{m,k,j}\}$, where $m=0,1,\cdots$ and $k=0,\cdots, N(\pi^n)$ and there exists $M<\infty$ such that $j<M$.
Then the function $h_{m,k,j}:[0,1]\to \mathbb{R}$ for all $m,k$ and $j$ can be expressed as: 

\begin{equation}\label{haar_basis_dyadic}
h_{m,k,j}(t)= 
\begin{cases}
    \quad 0 &\quad\text{if } t\notin \left[t^{m+1}_{p(m,k)},t_{p(m,k)+j}^{m+1}\right) \\
   
    \left( \frac{t^{m+1}_{p(m,k)+j}-t^{m+1}_{p(m,k)+j-1}}{t^{m+1}_{p(m,k)+j-1} - t^{m+1}_{p(m,k)}}\times\frac{1}{t^{m+1}_{p(m,k)+j}-t^{m+1}_{p(m,k)}} \right)^{\frac{1}{2}} &\quad\text{if } t\in\left[t_{p(m,k)}^{m+1},t_{p(m,k)+j}^{m+1}\right) \\
   
    -\left( \frac{t^{m+1}_{p(m,k)+j-1} - t^{m+1}_{p(m,k)}}{t^{m+1}_{p(m,k)+j}-t^{m+1}_{p(m,k)+j-1}}\times\frac{1}{t^{m+1}_{p(m,k)+j}-t^{m+1}_{p(m,k)}}\right)^{\frac{1}{2}} &\quad\text{if } t\in \left[t_{p(m,k)+j-1}^{m+1},t_{p(m,k)+j}^{m+1}\right),
   \end{cases} 
\end{equation}
where, $p(m,k)$ is defined in Equation \eqref{EqFor_p}. The non-uniform Haar basis $\{h_{m,k,j}\}$ is an orthogonal basis in one dimension. For convenience, reorder the non-uniform Haar basis to $\{h_{m,k}\}$, where $m=0,1,\cdots$ and $k= 0,1,\cdots, N(\pi^{m+1})-N(\pi^m)-1$.
Now we will define $d$-dimensional non-uniform Haar basis in the canonical way. Define $\{h_{m,k}^i\}$ for all $m=0,1,\cdots$, $k= 0,1,\cdots, N(\pi^{m+1})-N(\pi^m)-1$ and $i=1,2,\cdots, d$ as following. 
\begin{equation}
h_{m,k}^i(t):[0,1]\to \mathbb{R}^d \quad\text{ such that, } \quad h_{m,k}^i(t) = h_{m,k}(t) \times \textbf{e}_i,  
\end{equation}
where, $\textbf{e}_i$ is a $d$-dimensional column vector with $1$ at $i^{th}$ entry and $0$ elsewhere. Clearly, $\{\textbf{e}_i\}_{i=1,\cdots,d}$ is an orthogonal basis of $\mathbb{R}^d$. Denote $\textbf{0}$ to be a $d$-dimensional column vector with all entry as $0$. Now from the definition of $h_{m,k}^i$ we get
\[\int_0^1 h_{m,k}^i =\textbf{0} ; \quad \int_0^1 <h_{m,k}^i,h_{m,k}^i> = \textbf{e}_i \quad \text{and,} \int_0^1 <h_{m,k}^i,h_{m',k'}^j> =\mathbbm{1}_{i=j}\mathbbm{1}_{m=m'}\mathbbm{1}_{k=k'}\textbf{e}_i.\]
So $\{h_{m,k}^i\}$, where $m=0,1,\cdots$, $k=0,1,\cdots,N(\pi^{m+1})-N(\pi^m)-1$ and $i=1,\cdots ,d$ form an orthonormal basis in $\mathbb{R}^d$.
The  Schauder basis $e^{\pi,i}_{m,k}:[0,1]\to \mathbb{R}^d$ is defined as $e^i_{m,k}(t) =\left(\int_0^t h_{m,k}(u)du\right)\textbf{e}_i$ for $m\in \mathbb{N},$ $k=0,1,\cdots,N(\pi^{m+1}-N(\pi^m)-1)$ and $i=1,\cdots, d$.
\par The following theorem shows that any $d$-dimensional continuous function can be represented uniquely wrt the $d$-dimensional non-uniform  Schauder system associated with a finitely refining partition sequence.
\begin{theorem}
Let $\pi$ be a finitely refining sequence of partitions of $[0,1]$. Then any continuous function $x= \left(x_1,x_2,\cdots,x_d\right)\in C^0([0,1],\mathbb{R}^d)$ has a  {unique}  Schauder representation associated with $\pi$:
\[   \; x_i(t)= x_i(0)+ (x_i(1)-x_i(0))t +\sum_{m=0}^{\infty} \sum_{k=0}^{N(\pi^{m+1})-N(\pi^{m})-1} \theta_{m,k}^{\tcircle{i}} e^{\pi}_{m,k}(t).\]
If the support of the function $e_{m,k}^\pi$ is $[t^{m,k}_1,t^{m,k}_3]$ and its maximum is attained at time $t^{m,k}_2$ then 
\[\forall i=1,\cdots,d \qquad\theta_{m,k}^{\tcircle{i}} = \frac{\bigg[(x_i(t^{m,k}_{2})-x_i(t^{m,k}_{1}))(t^{m,k}_{3}-t^{m,k}_{2})-(x_i(t^{m,k}_{3})- x_i(t^{m,k}_{2}))(t^{m,k}_{2}-t^{m,k}_1) \bigg]}{\sqrt{(t^{m,k}_2-t^{m,k}_1)(t^{m,k}_3-t^{m,k}_2)(t^{m,k}_3-t^{m,k}_1)}}. \]
\end{theorem}
\begin{proof}
The proof is a straightforward extension of the one-dimensional case in Theorem \ref{coeff_hat_func}.
\end{proof}
We now give a multi-dimensional version of  Theorem \ref{QV-weighted-schied}.

\begin{theorem}\label{QV-weighted-schied_higherDim}
Let $\pi$ be a finitely refining sequence of partitions  of $[0,1]$ with vanishing mesh and 
\[   x_i(t)= x_i(0)+ (x_i(1)-x_i(0))t +\sum_{m=0}^{\infty} \sum_{k=0}^{N(\pi^{m+1})-N(\pi^{m})-1} \theta_{m,k}^{\tcircle{i}} e^{\pi}_{m,k}(t).\]
Then 
\ba 
[x_i,x_j]_{\pi^n}(t) = \sum_{m=0}^{n-1} \sum_{k=0}^{N(\pi^{m+1})-N(\pi^m)-1} a_{m,k}^n(t)\theta_{m,k}^{\tcircle{i}}\theta_{m,k}^{\tcircle{j}}  +\sum_{m,m'} \sum_{\substack{k,k'\\(m,k)\neq (m',k')}}b^n_{m,k,m',k'}(t)\theta_{m,k}^{\tcircle{i}}\theta_{m',k'}^{\tcircle{j}} . \qquad\ea 
If $[t^{m,k}_1,t^{m,k}_3]$ is the support of the function $e_{m,k}^\pi$, $t^{m,k}_2$ is maximum and
$\Delta t^n_i = t^n_{i+1}\wedge t- t^n_{i}\wedge t,$ then: 
\[ a_{m,k}^n(t) =\left\{ \left[ \sum_{\Delta t^n_i \subset [t^{m,k}_{1}, t^{m,k}_{2}]} (\Delta t^n_i)^2\right]\times \frac{t^{m,k}_{3}-t^{m,k}_{2}}{t^{m,k}_{2}-t^{m,k}_{1}} + \left[ \sum_{\Delta t^n_i \subset [t^{m,k}_{2}, t^{m,k}_{3}]} (\Delta t^n_i)^2\right]\times \frac{t^{m,k}_{2}-t^{m,k}_{1}}{t^{m,k}_{3}-t^{m,k}_{2}}\right\} \times \frac{1}{t^{m,k}_{3}-t^{m,k}_{1}}, \]
and,
\[b^n_{m,k,m',k'}(t) = \psi_{m',k'}(t^{m,k}_1)\times \left\{ \frac{ \sum_{\Delta t^n_i \subset [t^{m,k}_{1}, t^{m,k}_{2}]} (\Delta t^n_i)^2}{t^{m,k}_{2}-t^{m,k}_{1}} - \frac{\sum_{\Delta t^n_i \subset [t^{m,k}_{2}, t^{m,k}_{3}]} (\Delta t^n_i)^2}{t^{m,k}_{3}-t^{m,k}_{2}}\right\} \times \sqrt{\frac{(t^{m,k}_{2}-t^{m,k}_{1})(t^{m,k}_{3}-t^{m,k}_{2})}{t^{m,k}_{3}-t^{m,k}_{1}}}, \]
if ${\rm Supp}(e^n_{m,k})\subset {\rm Supp}(e^n_{m',k'})$, and $b^n_{m,k,m',k'}(t)=0$ if ${\rm Supp}(e^n_{m,k})\cap {\rm Supp}(e^n_{m',k'})=\emptyset$.
\end{theorem}
The following example is a 2-dimensional extension of the construction given  in Section \ref{sec5}.

\begin{example}\label{higher.dimn.Schid}[Example of process in 2 dimension with linear quadratic variation]
Define the class of processes $x\in C^0([0,T],\mathbb{R}^2)$ as following. For all $t\in [0,T]$: 
\begin{footnotesize}
\[ x(t) = \left( x_1(0)+ \left(x_1(1)-x_1(0) \right)t+ \sum_{m=0}^{\infty} \sum_{k=0}^{2^m-1} \theta_{m,k}^{\tcircle{1}} e_{m,k}(t),\qquad x_2(0)+ \left(x_2(1)-x_2(0) \right)t+ \sum_{m=0}^{\infty} \sum_{k=0}^{2^m-1} \theta_{m,k}^{\tcircle{2}} e_{m,k}(t) \right),\]
\end{footnotesize}
where, $\theta_{m,k}^{\tcircle{1}},\theta_{m,k}^{\tcircle{2}} \in \{-1,1\}$. This is the two dimensional extension of \cite{schied2016}. Then the quadratic variation of $x$ can be think of a $2\times 2$ matrix:
\[[x]_{T^n}(t)=\begin{pmatrix}
\frac{1}{2^n}\sum_{m=0}^{n-1} \sum_{k=0}^{2^m-1} (\theta_{m,k}^{\tcircle{1}})^2\mathbbm{1}_{[0,t]} & \frac{1}{2^n}\sum_{m=0}^{n-1} \sum_{k=0}^{2^m-1} (\theta_{m,k}^{\tcircle{1}})(\theta_{m,k}^{\tcircle{2}})\mathbbm{1}_{[0,t]} \\
\frac{1}{2^n}\sum_{m=0}^{n-1} \sum_{k=0}^{2^m-1} (\theta_{m,k}^{\tcircle{1}})(\theta_{m,k}^{\tcircle{2}})\mathbbm{1}_{[0,t]} & \frac{1}{2^n}\sum_{m=0}^{n-1} \sum_{k=0}^{2^m-1} (\theta_{m,k}^{\tcircle{2}})^2\mathbbm{1}_{[0,t]} 
\end{pmatrix}.\]
Since $\theta_{m,k}^{\tcircle{1}},\theta_{m,k}^{\tcircle{2}} \in \{-1,1\}$ we get $\frac{1}{2^n} \sum_{m=0}^{n-1} \sum_{k=0}^{2^m-1} (\theta_{m,k}^{\tcircle{1}})^2 = \frac{2^n-1}{2^n}\xrightarrow[]{n\to \infty} 1$, similarly,\\ $\frac{1}{2^n}\sum_{m=0}^{n-1} \sum_{k=0}^{2^m-1} (\theta_{m,k}^{\tcircle{2}})^2 \xrightarrow[]{n\to \infty}1$. 
\par If we further assume $\theta_{m,k}^{\tcircle{1}},\theta_{m,k}^{\tcircle{2}}$ are independent (not just uncorrelated) with $\mathbb{E}(\theta_{m,k}^{\tcircle{1}}) = 0$, then we get:
\[\mathbb{E}\left(\frac{1}{2^n}\sum_{m=0}^{n-1} \sum_{k=0}^{2^m-1} (\theta_{m,k}^{\tcircle{1}})(\theta_{m,k}^{\tcircle{2}}) \right) =\frac{1}{2^n}\sum_{m=0}^{n-1} \sum_{k=0}^{2^m-1} \mathbb{E}\left(\theta_{m,k}^{\tcircle{1}}\right)\mathbb{E}\left(\theta_{m,k}^{\tcircle{2}}\right) =0 \quad \text{ and,} \]
\[Var\left(\frac{1}{2^n}\sum_{m=0}^{n-1} \sum_{k=0}^{2^m-1} (\theta_{m,k}^{\tcircle{1}})(\theta_{m,k}^{\tcircle{2}}) \right) =\mathbb{E}\left(\frac{1}{2^n}\sum_{m=0}^{n-1} \sum_{k=0}^{2^m-1} (\theta_{m,k}^{\tcircle{1}})(\theta_{m,k}^{\tcircle{2}}) \right)^2 \]
\[= \left(\frac{1}{2^n}\right)^2\mathbb{E} \left[ \sum_{m=0}^{n-1} \sum_{k=0}^{2^m-1} \left((\theta_{m,k}^{\tcircle{1}})(\theta_{m,k}^{\tcircle{2}}) \right)^2+\underset{(m,k)\neq(m',k')}{\sum_{m=0}^{n-1} \sum_{k=0}^{2^m-1} \sum_{m'=0}^{n-1} \sum_{k'=0}^{2^{m'}-1}} \left((\theta_{m,k}^{\tcircle{1}})(\theta_{m,k}^{\tcircle{2}}) \right)\left((\theta_{m',k'}^{\tcircle{1}})(\theta_{m',k'}^{\tcircle{2}}) \right)\right] \]
\[= \left(\frac{1}{2^n}\right)^2 \left[ \sum_{m=0}^{n-1} \sum_{k=0}^{2^m-1} 1+\underset{(m,k)\neq(m',k')}{\sum_{m=0}^{n-1} \sum_{k=0}^{2^m-1} \sum_{m'=0}^{n-1} \sum_{k'=0}^{2^{m'}-1}} \mathbb{E}\left((\theta_{m,k}^{\tcircle{1}})(\theta_{m,k}^{\tcircle{2}}) \right)\mathbb{E}\left((\theta_{m',k'}^{\tcircle{1}})(\theta_{m',k'}^{\tcircle{2}}) \right)\right] \]
\[= \left(\frac{1}{2^n}\right)^2 \left[ \sum_{m=0}^{n-1} \sum_{k=0}^{2^m-1} 1\right] =\frac{1}{2^n} - \left(\frac{1}{2^n}\right)^2\xrightarrow[]{n\to \infty} 0. \]
We can see, $Var\left(\frac{1}{2^n}\sum_{m=0}^{n-1} \sum_{k=0}^{2^m-1} (\theta_{m,k}^{\tcircle{1}})(\theta_{m,k}^{\tcircle{2}}) \right)$ has an upper-bound of $\left(\frac{1}{2^n} - \frac{1}{2^{2n}}\right)$ which is the general term of a summable series. So using Borel-Cantelli lemma we can conclude $$\left(\frac{1}{2^n}\sum_{m=0}^{n-1} \sum_{k=0}^{2^m-1} (\theta_{m,k}^{\tcircle{1}})(\theta_{m,k}^{\tcircle{2}}) \right) \to 0 \quad \text{almost surely}.$$ So $[x]_{T^n}(t) \to t\textbf{Id}_{2\times 2}$ almost surely.
\end{example}
\begin{remark}
In general, the process we described in Example \ref{higher.dimn.Schid} is a process where the quadratic variation is linear over time along dyadic partition sequence, so they have the same quadratic variation as of two-dimensional Brownian Motion. But in contrast with Brownian paths (which belong to $ C^{\frac{1}{2}-}([0,1],\mathbb{R}^2)$) this process is  $\frac{1}{2}$-H\"older continuous. 
\par If we take $\theta_{m,k}^{\tcircle{2}} = 1$ in Example \ref{higher.dimn.Schid}, then the corresponding process $x$ has different quadratic variation along Triadic partition than that of 2-dimensional Brownian motion. 
\par If we take $\theta_{m,k}^{\tcircle{2}}$ and $\theta_{m,k}^{\tcircle{2}}$ are independent and $+1$ and $-1$ both with probability $\frac{1}{2}$ in Example \ref{higher.dimn.Schid}, then the process $x$ has the same quadratic variation along any finitely refining sequence of partitions which is coarsening of dyadic partition. This is a higher-dimensional extension of the process we discussed in Section \ref{sec5}. We have skipped the proof of this argument, as it follows in the similar line of the proofs discussed in Section \ref{sec5}.
\end{remark}

\bibliographystyle{siam}
\bibliography{pathwise1} 

\end{document}